\documentclass{jgcc} 
\pdfoutput=1
 \usepackage{lastpage}
 \jgccdoi{13}{2}{1}{7617}  
 \jgccheading{}{\pageref{LastPage}}{}{}%
 {Jun.~23,~2021}{Nov.~6,~2021}{}

\usepackage[numbers]{natbib}
\usepackage{amsthm,amssymb,amsmath,amsfonts,calrsfs,upgreek}
\usepackage{caption,subcaption}
\usepackage{tikz}
\usetikzlibrary{automata, positioning, arrows,decorations.markings, decorations.pathreplacing,angles,quotes}
\tikzset{
    node distance=3cm, 
    every state/.style={thick}, 
    every edge/.append style={line width=0.25mm}, 
    initial text=$ $, 
    }
\usepackage{hyperref}
\newcommand{\I}{\mathcal{I}}
\newcommand{\F}{\mathcal{F}}

\newcommand{\e}{\epsilon}
\newcommand{\A}{\mathcal{A}}
\newcommand{\D}{\mathcal{D}}

\newcommand{\R}{\mathbb{R}}
\newcommand{\initial}{\mathbf{i}}
\newcommand{\final}{\mathbf{f}}
\theoremstyle{plain} 

\begin{document}

\title[F.g. subgroups of free gp.-s as formal languages and their cogrowth]{Finitely generated subgroups of free groups as formal languages and their cogrowth}

\author[A.~Darbinyan]{Arman Darbinyan}	
\address{Department of Mathematics, Texas A\&M University, College Station, TX, USA}	
\email{adarbina@math.tamu.edu}  

\author[R.~Grigorchuk]{Rostislav Grigorchuk}	
\address{Department of Mathematics, Texas A\&M University, College Station, TX, USA}	
\email{grigorch@math.tamu.edu}  

\author[A.~Shaikh]{Asif Shaikh}	
\address{Department of Mathematics, Texas A\&M University, College Station, TX, USA \& Department of Mathematics \& Statistics, R. A. Podar College of Commerce and Economics, Mumbai, India}	
\email{asif.shaikh@rapodar.ac.in}  





\begin{abstract}
  \noindent Let $H\leq F_m$ be a finitely generated subgroup of the free group $F_m$ of finite rank $m$. It is well-known that the language $L_H$ of reduced words from $F_m$ representing elements of $H$ is regular.  In the current article, using the (extended) core of the Schreier graph of $H$, we construct the minimal deterministic finite automaton that recognizes $L_H$. Then we characterize the finitely generated subgroups $H\leq F_m$ for which $L_H$ is irreducible, and for each such $H$ we explicitly construct an ergodic automaton that recognizes $L_H$. This construction gives us an efficient way to compute the cogrowth series $L_H(z)$ as well as the entropy of $L_H$. Several examples are provided to illustrate the method. Also, a comparison is made with the method of calculation of $L_H(z)$ based on the use of Nielsen system of generators of $H$.
\end{abstract}

\maketitle
\section{Introduction}\label{S:intro}
In \cite{gri1978thesis, MR552478gri1979,MR599539Gri1980}, the notion of {\it cogrowth} of a subgroup $H$ of a free group $F_m$ was introduced  and a cogrowth criterion for amenability of the factor group $F_m/H$ (in case $H$ is normal) was proved. The concept of cogrowth was used to construct  counterexamples concerning various versions of the von Neumann conjecture about the existence of an invariant mean on groups and homogeneous spaces \cite{MR682486adyan1982,MR552478gri1979,ol1980problem}.

In \cite{gri1978thesis,MR552478gri1979,MR599539Gri1980}, it was observed that when $H \leq F_m$ is a finitely generated (f.g.) subgroup, then the cogrowth series 
\begin{equation}
    H(z) = \displaystyle \sum_{w \in H} z^{|w|} = \sum_{n=0}^{\infty} |H_n| z^n\label{eqn:Hcogrowth_def}
\end{equation}
is a rational function, where $|w|$ denotes the length of $w\in H$ and $H_n$ denotes the set of the reduced elements of length $n$ in $H$ with respect to a fixed basis of $F_m$.

Also, by \cite{gri1978thesis}, for  $H \leq F_m$,  the following formula holds:
\begin{equation}
    u(z) = \displaystyle R\left(z,\sqrt{1-\frac{2m-1}{m^2}z^2}\right)H\left( \frac{m\left(1-\sqrt{1-\frac{2m-1}{m^2}z^2}\right)}{(2m-1)z} \right) \nonumber,
\end{equation}
where $R(x,y)$ is a rational function explicitly determined in \cite{gri1978thesis}, and 
$$u(z) = \displaystyle \sum_{n=1}^{\infty} P_{1,1}^{(n)} z^n,$$ where $P_{1,1}^{(n)}$ is the probability of return to the identity $1 \in G = F_m/H$ in the simple random walk on $G$ that starts at the identity. This formula shows a close relation between analytic properties of functions $H(z)$ and $u(z)$. In particular, the algebraicity of $u(z)$ is equivalent to the algebraicity of $H(z)$. If $H$ is finitely generated, then the fact that $H(z)$ is rational was proven in \cite{MR599539Gri1980} using Nielsen system of generators for $H$. The topics related to growth and cogrowth have gotten a lot of attention and popularity, and are widely presented in the literature. See, for example, \cite{MR1436550griha1997}.

Among various open questions related to cogrowth, the authors suggest a conjecture that $H(z)$ is rational if and only if $H$ is finitely generated. Note that when $H \triangleleft F_m$ is a normal subgroup, the conjecture follows from the result of D. Kouksov \cite{MR1487319kuksov1998}.

The alternative approach to proving the rationality of $H(z)$ is via the theory of formal languages. Recall that the classical Chomsky hierarchy of languages begins with the class of regular (also called rational) languages, that is, languages recognizable by finite-automata acceptors \cite{MR645539Hopcroft1979}. Already in the 1950-60's, Chomsky and Sch\"utzenberger  were aware that the rationality of a language $L \subset \Sigma^*$ (where $\Sigma$ is finite alphabet and $\Sigma^*$ denotes the set of finite words over $\Sigma$) implies rationality of the growth series
$$L(z) = \sum_{n=0}^{\infty} |{B_n(L)}| z^n,$$ where ${B_n(L)}$ is the set of words in $L$ of length $n$. In fact, back then it was more popular to consider the following noncommutative version of the growth series, $$\hat{L} = \sum_{w \in L} w,$$ the rationality of which is equivalent to the rationality of $L$ \cite{MR2760561berstel2011, MR0483721saloma1978}. A  concept closely related to cogrowth is the 
\emph{entropy} of languages, which is defined as $$h(L) = \displaystyle\limsup_{n \rightarrow \infty}\frac{\log |{B_n(L)}| }{n}.$$

For a fixed free basis $A = \{a_1,\cdots,a_m\}$  of $F_m$,  we denote by $L_H$ the set of all reduced finite words in $(A\cup A^{-1})^*$ that represent an element of $H\leq F_m$. The fact that regularity of the language $L_H$ is equivalent to the finite generation of $H$ was observed by Anissimov and Seifert \cite{Anissimow1975ZurAC}. Since the intersection of two regular languages is regular (see \cite{MR645539Hopcroft1979}), a direct consequence of the theorem of Anissimov and Seifert is that the intersection of two finitely generated subgroups of $F_m$ is again finitely generated. The last observation is in fact  a well-known theorem of Howson from 1954, \cite{MR65557howson1954}. A proof of the theorem of Anissimov and Seifer, based on the ideas of geometric group theory, is presented in \cite{MR1114609gersten1991}. In fact, there are several ways to prove that if $H \leq F_m$ is finitely generated, then $L_H$ is a regular language. An elegant proof of this statement is presented by I. Kapovich and A. Myasnikov in \cite{MR1882114kapovich2002}. Their proof is based on the idea of J. Stallings from \cite{MR695906Stallings1983}, which is now known as {\it Stallings foldings}, and on the notion of the core of a Schreier graph associated with the triple $(F_m, H, A).$ One of the goals of the current article is to explore this approach from the finite automata theory point of view and make it more detailed and accessible.

We conclude the introduction by outlining the content of the rest of the paper. In Section 2, we recall some of the basic definitions and terminology from the theory of finite automata, formal languages and theory of graphs that will be needed later. In Section 3, we give two versions of the definition of Schreier graph, and hence two versions of the core, one of which we call \emph{the extended core}. We also discuss a combinatorial procedure to obtain the (extended) core of $H\leq F_m$. In Section 4, we recall the definition and properties of the Nielsen system of generators of $H \leq F_m$. We also recall how to obtain a Nielsen system geometrically. In Section 5, using the DFA $\A_{F_m}$ that recognizes the language of freely reduced words of $F_m$ and the DFA $\A_{\widehat{\Delta}_H}$ whose Moore diagram is the extended core $\widehat{\Delta}_H$, we build another automaton, $\A_H$, that recognizes the language $L_H$ of freely reduced words from $F_m$ that represent the elements in $H$. Then we extract from $\A_H$ a  minimal sub-automaton $\D_H$  such that $L(\D_H) = L_H$. In Subsection 5.4, we characterize the f.g. subgroups $H \leq F_m$ for which $L_H$ is irreducible (see Proposition \ref{prop-when-L_H-is-irreducibl}) and for such $H$ we explicitly construct ergodic automaton $\widehat{\D}_H$ that recognizes $L_H$ (see Theorem \ref{thm-ergodicity}). The ergodicity of $\widehat{\D}_H$ allows us, by applying Perron-Frobenius theory, to obtain in Subsection 5.5 a matrix formula for the entropy of $L_H$, allowing its efficient computation (see Theorem \ref{thm-on-computing-entropy}). One more ingredient for showing this formula is based on relating to each finite state automaton so called \emph{base automaton}, and by showing that under certain natural restrictions on the automaton, its entropy coincides with the entropy of its base automaton (see, Proposition \ref{prop-aux}).

The standard {\it transfer matrix} method, that goes back to Kolmogorov's theory of finite Markov chains, leads to the system of linear equations that allows us to compute $H(z)$. All of the above, in principle, is applicable to arbitrary subgroup $H \leq F_m$, but the related automata and the system of equations are finite only when $H$ is finitely generated. In Section 6, we also include reproduction of the proof from \cite{MR599539Gri1980} of rationality of $H(z)$ via Nielsen system of generators. The produced algorithms are of polynomial complexity. An interesting question for further investigation is to check which of the two approaches on a given finite set $\{w_1,\cdots,w_k\}$ of generators of $H$ is more efficient for computing $H(z)$.

Finally, in the last section, several concrete examples illustrating the theoretical part from the previous sections are included.

\section{Preliminaries}\label{S:preli}
Throughout this paper, by $A=\{a_1, \ldots, a_m\}$ we denote a fixed basis of the free group $F_m$, elements of which, along with their inverses, we regard as formal letters, whenever they are considered in the context of formal languages. Correspondingly, the set of generators $\Sigma = A \cup A^{-1}$ of $F_m$ will be regarded as an alphabet whenever it is in the context of formal languages.
 The set of all finite words over the alphabet $\Sigma$ is denoted by $\Sigma^*$. Algebraically, $\Sigma^*$ is the free monoid generated by the finite set $\Sigma$. The {\it length} of a word $w \in \Sigma$, denoted by $|w|$, means the number of letters in $w$ when each letter is counted as many times as it occurs. By $red(w)$ we denote the word that is obtained from $w \in \Sigma$ by free reduction. The subsets of $\Sigma^*$ are referred to as (formal) languages over the alphabet $\Sigma$. A language $L$ is called {\it regular} if it is recognized by some finite automaton. A finite automaton $\A$ is a quintuple, $\A = (Q,\Sigma,\delta,\I,\F)$, consisting of finite set of states $Q$, alphabet $\Sigma$, transition function $\delta:Q\times\Sigma \rightarrow 2^Q$, the set of initial states $\I \subseteq Q$ and the set of final states $\F \subseteq Q$. 

Let $G_{\A}$ be the {\it Moore} or transition diagram of $\A$, that is $G_{\A}$ is a labelled directed graph with vertex set $Q$ and the directed labelled edges are described by the transition function  $\delta$ with labels from $\Sigma$. Namely, vertex $q$ is connected with vertex $q'$ with an edge labeled by $a \in \Sigma$, if $q' \in \delta(q, a)$. (For example, Figure \ref{fig:A_F2} is a depiction of a Moore diagram for an automaton that we define later.) 

Let $e$ be an edge in $G_{\A}$. Then by $o(e), t(e)$ and $l(e)$ we denote origin, terminus and the label of the edge $e$, respectively. A directed path $p=e_1\cdots e_n$ in $G_{\A}$ is called \emph{admissible} if $o(e_1) \in \I, t(e_i) = o(e_{i+1}),$ for $i = 1,\cdots,(n-1)$, $t(e_n) \in \F$. Let $w = x_1\cdots x_n$ be a word over $\Sigma$. The automaton $\A$ accepts the word $w$ if there is an admissible path $p$ in $G_{\A}$ such that $l(p) = l(e_1)\cdots l(e_n) = w$. The set of words that $\A$ accepts is the language recognized by $\A$ and this language is denoted by $L(\A)$. 

An automaton $\A$ is \emph{ergodic} if its Moore diagram $G_{\A}$ is strongly connected, that is, for any two states $q$ and $q' \in Q$ there exists a path connecting $q$ to $q'$. A language $L \subseteq \Sigma^*$ is \emph{irreducible} if, given two words $w_1, w_2 \in L$, there exists a word $w \in \Sigma^*$ such that the concatenation $w_1ww_2 \in L$. A regular language $L$ is irreducible if and only if it is generated by some ergodic automaton, see Theorem 3.3.11 of \cite{Lind-Marcus}. An automaton $\A$ is \emph{unambiguous} if for every $w \in L(\A)$, there is a unique admissible path $p \in G_{\A}$ such that $l(p) = w$. An automaton $\A$ is \emph{deterministic}, if for each state of $Q$, all outgoing edges carry distinct labels. It is obvious that a deterministic automaton with one initial state is unambiguous. Note that $\A$ is deterministic if the codomain of $\delta$ is $\{\emptyset\} \cup Q$, that is $\delta:Q\times\Sigma \rightarrow \{\emptyset\} \cup Q$. The automaton $\A' = (Q',\Sigma',\delta',\I',\F')$ is said to be a \emph{subautomaton} of $\A=(Q,\Sigma,\delta,\I,\F)$, if
$Q' \subseteq Q$, $\Sigma' \subseteq \Sigma$, $\I' \subseteq \I$, $\F' \subseteq \F$, and for each $q' \in Q'$, $a' \in \Sigma'$, $\delta'(q', a') = \delta(q, a) \cap Q'$. In the language of Moore diagrams, an equivalent definition would be: $\A'$ is a subautomaton of $\A$ if its Moore diagram  $G_{\A'}$ is a subdiagram of $G_{\A}$.

Let $\A$ be a finite automaton. We say that $\A$ is \emph{essential} if in its Moore diagram $G_{\A}$ every vertex (hence, also every edge) belongs to some path connecting an initial state to a final state, i.e. to an admissible path. If $\A'$ is an essential subautomaton of $\A$ such that $L(\A')=L(\A)$, then we say that $\A'$ is an \emph{essential part} of $\A$.

Let $k \geq 1$. An automaton $\A$ has \emph{homogeneous ambiguity} $k$ if, for any nonempty word $w \in L(\A)$, there are exactly $k$ admissible paths $p_1,\cdots,p_k$ in $G_{\A}$ with label $w$. In case the number of such paths for each $w\in L(\A)$ is bounded from above by $k$, we say that $\A$ has \emph{bounded ambiguity}. We shall use the terminology \emph{DFA} $\A$ for the deterministic (unambiguous) finite automaton with exactly one initial state i.e. $\A = (Q,\Sigma,\delta,\{q_0\},\F)$. The DFA $\A$ is said to be \emph{minimal} if there is no DFA with smaller number of states that recognizes the same language $L(\A)$. It is known that for any regular language $L$, up to isomorphism, there is a unique minimal DFA recognizing $L$, called \emph{the minimal deterministic finite automaton} of $L$  (see Theorem 3.10 on page 67 of \cite{MR645539Hopcroft1979}).

By \emph{inaccessible state} we mean a non-initial state that does not have any incoming edge. Let $\A$ be a DFA and let $w, w' \in \Sigma^*$. For every state $q$ and word $w$, the value of $\delta(q, w)$ is the end state of the path in $\A$ that starts at $q$ and reads the input word $w$. In case such path does not exist, we define $\delta(q, w)=\emptyset$. 
Let $L = L(\A)$ be the language recognized by DFA $\A$ and let $w, w' \in \Sigma^*$. Then there is a natural equivalence relation $R_L$ on words associated with $L$, given by
\begin{equation}
    wR_Lw' \iff \forall~z \in \Sigma^* \textnormal{ either both or neither of } wz \textnormal{ and } w'z \textnormal{ is in } L \label{eqn:relation_R_L}
\end{equation}
Observe that the number of equivalence classes of $R_L$ is at most the number of states of $\A$, which is finite. Now we recall a version of Myhill-Nerode Theorem.
{\thm \label{NERODE:Theorem} (Theorem 3.9 and Theorem 3.10 of \cite{MR645539Hopcroft1979}) Let $L \subseteq \Sigma^*$ be a regular language. Then, the relation $R_L$ defines a DFA $\A' = (Q',\Sigma,\delta',\{q_0'\},\F')$ for $L$ whose states correspond to the equivalence classes of $R_L$. Moreover, this is the unique minimal DFA for $L$ (up to isomorphism), where 
$$Q' = \{\left[w\right]\mid  w \in \Sigma^*\}$$
$$\{q_0'\} = \textnormal{the equivalence class of the empty word}.$$
$$\F' = \{\left[w\right]\mid  w \in L\} \textnormal{ and }$$
$$\delta'(\left[w\right],a) = \left[wa\right].$$}
A graph $G = (V,E)$ is called locally finite if the degree of each vertex of $V$ is finite. The $in$ degree of a vertex $v$ of the directed graph is the number of edges in the graph that have $v$ as the terminus. Similarly the $out$ degree of a vertex $v$ of the graph is the number of edges in the graph that have $v$ as the origin. We denote $in$ and $out$ degree of the vertex $v$ by $\deg^-(v)$ and $\deg^+(v).$ If for every vertex $v$ of the graph $\deg^+(v) = \deg^-(v)$, then we will ignore the signs and by $\deg(v)$ we denote the degree of the vertex $v$. For further details on the theory of finite automata, we refer the reader to \cite{MR645539Hopcroft1979} and \cite{MR0483721saloma1978}. 
\section{The Schreier graph and the core associated with a subgroup $H$ of $F_m$}\label{S:3_Schr}
This section is devoted to the Schreier graph and the core of a f.g. subgroup $H$ of $F_m$.  We shall also discuss the procedure to obtain the core of $H$ using Stallings foldings and some of the important properties of the core.
\subsection{The Schreier graph of subgroup $H$ of $F_m$}\label{S:3.1_Schr}
We define two versions of the Schreier graph associated with $H \leq F_m$, which we denote by $\Gamma$ and $\widehat{\Gamma}$, respectively. The set of vertices of $\Gamma$ and $\widehat{\Gamma}$ is the same and is the set $V = \{ H_g \mid g \in F_m\}$ of right cosets. The set of edges $E$ of $\Gamma$ is the set $E = \{(H_g, H_{ga})\mid  g \in F_m, a \in A\}$ consisting of pairs $e = (H_g, H_{ga})$ of cosets. The edges are oriented and $H_g$ is the origin $o(e)$ of $e$ while $H_{ga}$ is the terminus $t(e)$ of $e$. Moreover, such an edge has the label $\mu(e) = a.$ Each vertex in $\Gamma$ has $m$ outgoing edges whose labels constitute the set $A$. The graph $\widehat{\Gamma}$ is obtained from $\Gamma$ by adding edges from the set $\overline{E} = \{\overline{e}\mid e \in E\}$ where $\overline{e} = (H_{ga},H_g)$ if $e = (H_g, H_{ga})$ and the label $\mu(\overline{e}) = \mu(e)^{-1} = a^{-1} \in A^{-1}.$ Thus $\Gamma = (V, E, \mu)$ and $\widehat{\Gamma} = (V,E\cup \overline{E},\widehat{\mu})$, where $\widehat{\mu}(e) = \mu(e)$ if $e \in E$ and $\widehat{\mu}(\bar e) = \mu(e)^{-1}$ if $\bar e \in \overline{E}$. Each vertex of $\widehat{\Gamma}$ has $2m$ outgoing edges and $2m$ incoming edges, whose labels constitute the set $\Sigma = A \cup A^{-1}$. We call $\Gamma$ the \emph{Schreier graph} and $\widehat{\Gamma}$ the \emph{extended Schreier graph} of $H$. The vertex $v_0 = H_1 = H$ is the distinguished vertex, so in fact $\Gamma$ and $\widehat{\Gamma}$ are rooted graphs with root $v_0$. Observe that in fact, according to the standard terminology in graph theory, $\Gamma$ and $\widehat{\Gamma}$ are multigraphs as they may have loops and multiple edges. We will use the obvious notion of path $p$ in directed graph $\Gamma$ or $\widehat{\Gamma}$ and its label $\mu(p) \in A^*$ or $\widehat{\mu}(p) \in \Sigma^*$, respectively.
\subsection{The core graph of subgroup $H$ of $F_m$}\label{S:3.2_Core}
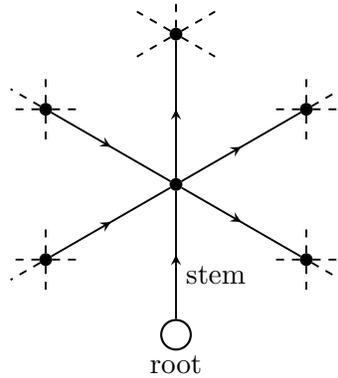
\begin{figure}[!htb]
    \centering
    \begin{tikzpicture}[scale=1,decoration={markings, mark= at position 0.5 with {\arrow{stealth}}}]
\tikzstyle{knode}=[circle,draw=black,thick,text width = 1.5 pt,align=center,inner sep=1pt]
\tikzstyle{rnode}=[circle,draw=black,thick,text width = 9 pt,align=center,inner sep=1pt]
\node (q1) at (0,-2) [rnode] {};
\draw[] (q1) node[below=4pt] {root};
\draw[] (0,-1.2) node[right] {  stem};
\node[fill] (q2) at (0,0) [knode] {};
\node[fill] (q3) at (-1.7320,-1) [knode] {};
\node[fill] (q4) at (-1.7320,1) [knode] {};
\node[fill] (q5) at (0,2) [knode] {};
\node[fill] (q6) at (1.7320,1) [knode] {};
\node[fill] (q7) at (1.7320,-1) [knode] {};
\draw (q1) edge[postaction={decorate}] (q2)
      (q3) edge[postaction={decorate}] (q2)
      (q4) edge[postaction={decorate}] (q2)
      (q2) edge[postaction={decorate}] (q5)
      (q2) edge[postaction={decorate}] (q6)
      (q2) edge[postaction={decorate}] (q7);
\draw[-,dashed] (q5) edge (0,2.5)
        (q5) edge (1.7320*0.3,1.7)
        (q5) edge (-1.7320*0.3,1.7)
        (q5) edge (1.7320*0.3,2.3)
        (q5) edge (-1.7320*0.3,2.3);
\draw[-,dashed] (q3) edge (-2.2,-2.2*0.5773)
       (q3) edge (-1.7320,-.5)
       (q3) edge (-1.7320,-1.5)
       (q3) edge (-2.232,-1)
       (q3) edge (-1.232,-1);
\draw[-,dashed] (q4) edge (-2.2,2.2*0.5773)
       (q4) edge (-1.7320,1.5)
       (q4) edge (-1.7320,.5)
       (q4) edge (-2.232,1)
       (q4) edge (-1.232,1);
\draw[-,dashed] (q6) edge (2.2,2.2*0.5773)
       (q6) edge (1.7320,.5)
       (q6) edge (1.7320,1.5)
       (q6) edge (2.232,1)
       (q6) edge (1.232,1);
\draw[-,dashed] (q7) edge (2.2,-2.2*0.5773)
       (q7) edge (1.7320,-1.5)
       (q7) edge (1.7320,-.5)
       (q7) edge (2.232,-1)
       (q7) edge (1.232,-1);
\end{tikzpicture}
    \caption{A branch}
    \label{fig:branch}
\end{figure}
The \emph{core} $\Delta_H = (\widehat{V},E_{\Delta_H})$ is the subgraph of the Schreier graph $\Gamma$ that is defined as the union of closed paths containing the root vertex $v_0$. 

A branch of a $k$-regular tree is a subtree which has one degree 1 vertex, which we call the {\it root} of the branch and all the other vertices have degree $k$. Such a branch is uniquely determined by its stem, which is the oriented edge going from the root to the interior of the branch, see Figure (\ref{fig:branch}). A subgraph of the Schreier graph $\Gamma$ isomorphic to a branch in the Cayley graph $X_m$ of $F_m$ (with its labeling) is called a hanging branch.  The Cayley graph $X_m = Cay(F_m, \Sigma)$ is a homogeneous tree of degree $|\Sigma|.$ The core of the Schreier graph $\Gamma$ can be obtained also by removing the hanging branches. Moreover, if the core $\Delta_H$ is known, then the graph $\Gamma$ can be obtained from the core $\Delta_H$ by filling the deficient valencies of the vertices of $\Delta_H$ with maximal hanging branches (so that all the degrees of the resulting graph have the degree $2m$). Thus, since the Schreier graph $\Gamma$ is connected, its core $\Delta_H$ is also connected. We refer the reader to \cite{MR2921182gri_ergodic} for the descriptions of the Hopf decomposition of the boundary in terms of $\Gamma$, $\Delta_H$ and the collection of hanging branches.

Let $\overline{E}_{\Delta_H} = \left\{\overline{e}\mid e \in E_{\Delta_H}\right\}$. We now define the {\it extended core} graph $\widehat{\Delta}_H = \left(\widehat{V},\widehat{E},\widehat{\mu}\right)$ from the core $\Delta_H$, where $\widehat{E} = E_{\Delta_H}\cup\overline{E}_{\Delta_H}$. Observe that the extended core $\widehat{\Delta}_H$ is a subgraph of $\widehat{\Gamma}.$ It is easy to see that if $e \in \widehat{E}$ then $\overline{e} \in \widehat{E}$ (i.e. if $e$ belongs to the path $p$, then $\overline{e}$ belongs to the path $\overline{p}$ obtained from $p$ by obvious inversion of the direction). We say that a labeled path  is \emph{reduced} if it does not contain adjacent edges with labels of the form $aa^{-1}$, otherwise, we say that the path is not reduced or  we say that it \emph{backtracks}. Note that paths in the graph $\widehat{\Delta}_H$ are not necessarily reduced and may backtrack. For example, a path $p = e\overline{e}$ in $\widehat{\Delta}_H$ from $v$ to $v$, where $e, \overline{e} \in \widehat{E}$ and $o(e) =v = t(\overline{e}), t(e) = v' = o(\overline{e})$ is not a reduced path. In further applications we regard $\widehat{\Delta}_H$ as a DFA $\A_{\widehat{\Delta}_H}$ having the root vertex $v_0$ as the initial and the final states. More precisely, the Moore diagram of the DFA $$\A_{\widehat{\Delta}_H} = \left(\widehat{V}, \Sigma, \delta_{\widehat{\Delta}_H}, \{v_0\}, \{v_0\}\right)$$ is the extended core graph $$\widehat{\Delta}_H = \left(\widehat{V},\widehat{E},\widehat{\mu}\right), $$ where for each $e \in \widehat{E}$ connecting vertex $v$ to $va_j^{\e}$, we have
\begin{equation}
    \delta_{\widehat{\Delta}_H}(v,a_j^{\e}) = va_j^{\e},\label{eqn:extended_transition}
\end{equation}
where $\widehat{\mu}(e) = a_j^{\e} \in \Sigma.$ 
The language $L(\A_{\widehat{\Delta}_H})$ of the DFA $\A_{\widehat{\Delta}_h}$ contains words $w = \widehat{\mu}(p)$, where $p$ is a admissible path of $\A_{\widehat{\Delta}_H}$. Notice that the admissible paths $p$ in $\A_{\widehat{\Delta}_H}$ may or may not be reduced. Hence not all words in the language $L(\A_{\widehat{\Delta}_H})$ are reduced. We denote by $L_H$ the language of reduced elements of a f.g. subgroup $H$ of $F_m.$ Theorem 5.1 from \cite{MR1882114kapovich2002} can be read as
{\thm \label{Theorem:lang_extended_core} $L_H \subseteq L(\A_{\widehat{\Delta}_H})$. Moreover, the words in $L(\A_{\widehat{\Delta}_H})\setminus L_H  $  are not reduced.}

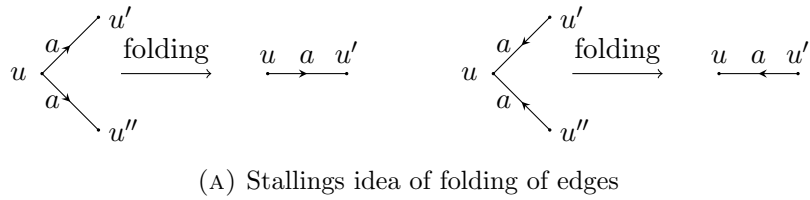
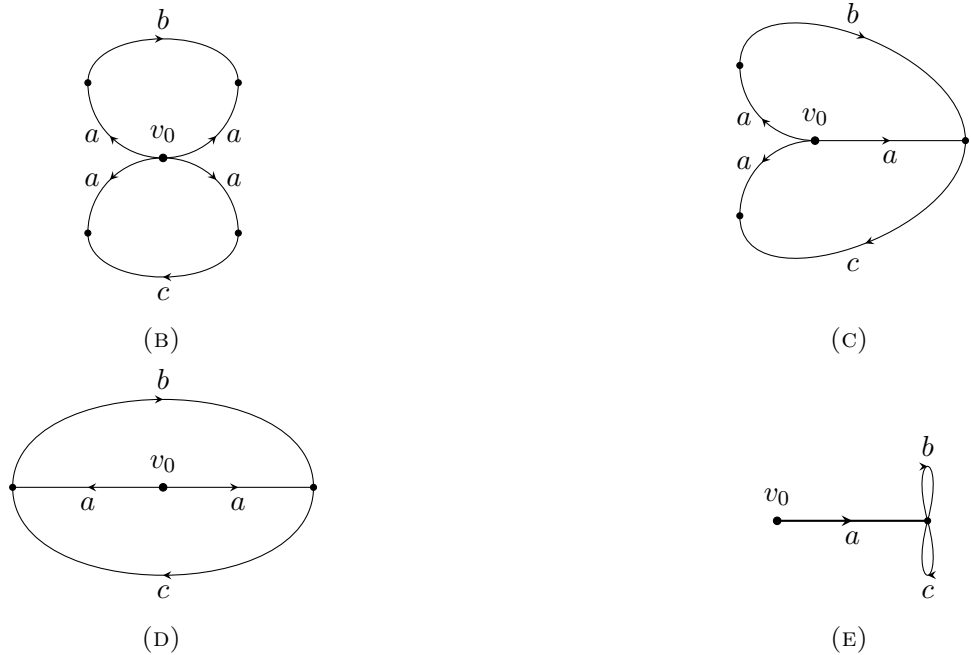
\begin{figure}[!htb]
\begin{subfigure}[b]{\textwidth}
        \centering
          \begin{tikzpicture}[scale=1.5,decoration={markings, mark= at position 0.5 with {\arrow{stealth}}}]
          \draw[fill] (0,0) node[left=2pt] {$u$} circle [radius=0.010];
          \draw[fill] (0.5,0.5) node[right] {$u'$} circle [radius=0.010];
          \draw[fill] (0.5,-0.5) node[right] {$u''$} circle [radius=0.010];
          \draw[fill] (2,0) node[above] {$u$} circle [radius=0.010];
          \draw[fill] (2.7,0) node[above] {$u'$} circle [radius=0.010];
          \draw[thin,postaction={decorate}] (0,0)  to [out=45,in=225] node[left]{$a$} (.5,.5) ;
          \draw[thin,postaction={decorate}] (0,0)  to [out=-45,in=135] node[left]{$a$} (.5,-.5) ;
          \draw[thin,->] (.7,0) to node[above]{folding} (1.5,0);
          \draw[thin,postaction={decorate}] (2,0)  to node[above]{$a$} (2.7,0) ;
          \draw[fill] (4,0) node[left=2pt] {$u$} circle [radius=0.010];
          \draw[fill] (4.5,0.5) node[right] {$u'$} circle [radius=0.010];
          \draw[fill] (4.5,-0.5) node[right] {$u''$} circle [radius=0.010];
          \draw[fill] (6,0) node[above] {$u$} circle [radius=0.010];
          \draw[fill] (6.7,0) node[above] {$u'$} circle [radius=0.010];
          \draw[thin,postaction={decorate}] (4.5,0.5)  to [in=45,out=225] node[left]{$a$} (4,0) ;
          \draw[thin,postaction={decorate}] (4.5,-0.5)  to [in=-45,out=135] node[left]{$a$} (4,0) ;
          \draw[thin,->] (4.7,0) to node[above]{folding} (5.5,0);
          \draw[thin,postaction={decorate}] (6.7,0)  to node[above]{$a$} (6,0) ;
          \end{tikzpicture}
          \caption{Stallings idea of folding of edges}
        \label{fig:folding_steps}
       \end{subfigure}
       \newline
    \begin{subfigure}[b]{0.4\textwidth}
        \centering
    \begin{tikzpicture}[scale=2,decoration={markings, mark= at position 0.5 with {\arrow{stealth}}}]
        \draw[fill] (0,0) node[above=2pt] {$v_0$} circle [radius=0.025];
        \draw[fill] (0.5,0.5) node[] {} circle [radius=0.02];
        \draw[fill] (-0.5,0.5) node[] {} circle [radius=0.02];
        \draw[fill] (-0.5,-0.5) node[] {} circle [radius=0.02];
        \draw[fill] (0.5,-0.5) node[] {} circle [radius=0.02];
        \draw[thin,postaction={decorate}] (0,0)  to [out=0,in=270] node[right]{$a$} (.5,.5) ;
    	\draw[thin,postaction={decorate}] (-.5,.5)  to [out=90,in=90] node[above]{$b$}(.5,.5) ;
    	\draw[thin,postaction={decorate}] (0,0)  to [out=180,in=270] node[left]{$a$} (-.5,.5) ;
    	\draw[thin,postaction={decorate}] (0,0)  to [out=0,in=90] node[right]{$a$} (.5,-.5) ;
    	\draw[thin,postaction={decorate}] (.5,-.5)  to [out=270,in=270] node[below]{$c$}(-.5,-.5) ;
    	\draw[thin,postaction={decorate}] (0,0)  to [out=180,in=90] node[left]{$a$} (-.5,-.5) ;
    \end{tikzpicture}
          \caption{}
        \label{fig:staling_wedge}
       \end{subfigure}
       \hfill
       \begin{subfigure}[b]{0.4\textwidth}
        \centering
    \begin{tikzpicture}[scale=2,decoration={markings, mark= at position 0.5 with {\arrow{stealth}}}]
        \draw[fill] (0,0) node[above=2pt] {$v_0$} circle [radius=0.025];
        \draw[fill] (1,0) node[] {} circle [radius=0.02];
        \draw[fill] (-0.5,0.5) node[] {} circle [radius=0.02];
        \draw[fill] (-0.5,-0.5) node[] {} circle [radius=0.02];
        \draw[thin,postaction={decorate}] (0,0)  to [out=0,in=180] node[below]{$a$} (1,0) ;
    	\draw[thin,postaction={decorate}] (-.5,.5)  to [out=90,in=90] node[above]{$b$}(1,0) ;
    	\draw[thin,postaction={decorate}] (0,0)  to [out=180,in=270] node[left]{$a$} (-.5,.5) ;
    	\draw[thin,postaction={decorate}] (1,0)  to [out=270,in=270] node[below]{$c$}(-.5,-.5) ;
    	\draw[thin,postaction={decorate}] (0,0)  to [out=180,in=90] node[left]{$a$} (-.5,-.5) ;
    \end{tikzpicture}
          \caption{}
        \label{fig:folding1}
       \end{subfigure}
       \newline
       \begin{subfigure}[b]{0.4\textwidth}
        \centering
    \begin{tikzpicture}[scale=2,decoration={markings, mark= at position 0.5 with {\arrow{stealth}}}]
        \draw[fill] (0,0) node[above=2pt] {$v_0$} circle [radius=0.025];
        \draw[fill] (1,0) node[] {} circle [radius=0.02];
        \draw[fill] (-1,0) node[] {} circle [radius=0.02];
        \draw[thin,postaction={decorate}] (0,0)  to [out=0,in=180] node[below]{$a$} (1,0) ;
    	\draw[thin,postaction={decorate}] (-1,0)  to [out=90,in=90] node[above]{$b$}(1,0) ;
    	\draw[thin,postaction={decorate}] (0,0)  to [out=180,in=0] node[below]{$a$} (-1,0) ;
    	\draw[thin,postaction={decorate}] (1,0)  to [out=270,in=270] node[below]{$c$}(-1,0) ;
    \end{tikzpicture}
          \caption{}
        \label{fig:folding2}
       \end{subfigure}
       \hfill
       \begin{subfigure}[b]{0.4\textwidth}
        \centering
    \begin{tikzpicture}[scale=2,decoration={markings, mark= at position 0.5 with {\arrow{stealth}}}]
        \draw[fill] (0,0) node[above=2pt] {$v_0$} circle [radius=0.025];
        \draw[fill] (1,0) node[] {} circle [radius=0.02];
        \draw[thick,postaction={decorate}] (0,0)  to [out=0,in=180] node[below]{$a$} (1,0) ;
    	\draw[thin,postaction={decorate}] (1,0)  to [loop above, above] node{$b$} (1,0) ;
    	\draw[thin,postaction={decorate}] (1,0)  to [loop below, below] node{$c$} (1,0) ;
    	
    \end{tikzpicture}
          \caption{}
        \label{fig:folded}
       \end{subfigure}
    \caption{Construction of the core $\Delta_H$ of $H = \langle aba^{-1}, aca^{-1}\rangle$. 
    In (\ref{fig:staling_wedge}), we start with a bouquet of two loops attached to a base vertex $v_0$. Each loop we split into 3 oriented edges with positive labels $a,b,a$ and $a,c,a$, respectively. Observe that the first two edges of each loop have clockwise orientation whereas the third edge in both the loops has the reverse orientation. In Figures (\ref{fig:folding1}) and (\ref{fig:folding2}), we use Stallings foldings (see \ref{fig:folding_steps}) to obtain the required $\Delta_H$ (see \ref{fig:folded}).}
    \label{fig:stallings}
\end{figure}
\subsection{Stallings foldings}\label{S:3.3_Stallings}
Let $H$ be generated by elements $w_1,\cdots,w_k$. We identify $w_i$ with freely reduced words in the alphabet $\Sigma$. Then there is a simple combinatorial procedure to obtain the core $\Delta_H$. This procedure is based on the topological idea of \emph{folding} developed by J. Stallings in \cite{MR695906Stallings1983}. Roughly the procedure can be described as follows.

Start with a bouquet of $k$ circles glued together along a vertex $v_0$. Split the $i$-th circle, $1\leq i \leq k$, into $|w_i|$ edges which are oriented and labeled by the letters from the set $A\cup A^{-1}$ so that the label of the $i$-th circle(as read from $v_0$ to $v_0$) is precisely the word $w_i$. Reverse the edges with (negative) label $x^{-1}$ from the set $A^{-1}$ and assign the (positive) label $x$ from the set $A$ (see Figure \ref{fig:stallings}).

Suppose $e_1, e_2$ are edges of this graph with a common origin and the same label $ x \in A$. Then, informally, folding the graph at $e_1, e_2$ means identifying $e_1$ and $e_2$ in a single new edge labeled by $x$. 

At the first step, fold the graph at the edges that are originated at the root vertex $v_0$. After performing these foldings, fold the graph at the edges that are originated at the other vertices and continue the process. As we assumed that the subgroup $H$ is finitely generated, the process will stop after applications of finitely many steps (i.e. it stops when no more folding is possible). The resulting graph, up to isomorphism of labeled graphs, does not depend on the performed sequence of foldings. And the resulting graph is isomorphic to the core $\Delta_H$ (see \cite{MR1882114kapovich2002, MR695906Stallings1983}). Moreover, the algorithm based on above procedure has polynomial time complexity (see \cite{MR1882114kapovich2002}).

In fact, the above procedure works also for the situation of an infinitely generated group $H = \langle w_1,\cdots,w_n,\cdots\rangle$. One just has to begin with folding of the bouquet of the two loops labeled by $w_1$ and $w_2$, after the process stops add to the obtained graph a new loop labeled by $w_3$, apply folding, then add $w_4$ etc. The process will converge to the infinite folded (i.e. no more folding is possible) graph $\Delta_H$ with rooted vertex $v_0$.

The following lemma lists some of the well-known properties of the graph $\widehat{\Delta}_H$ of $H \leq F_m$. 
\begin{lem} \label{lem:properties_core} Let $\widehat{\Delta}_H= \left(\widehat{V},\widehat{E},\widehat{\mu}\right)$ be the extended core of $H$. Then the following holds.
\begin{enumerate}
    \item \label{prop_1} The graph $\widehat{\Delta}_H$ has no degree one vertices, except possibly for the root vertex $v_0$.
    \item \label{prop_2} The degree of each vertex in $\widehat{\Delta}_H$ is at most $2m$.
    \item \label{prop_3} $[F_m:H] < \infty$ if and only if $\widehat{\Delta}_H$ is a finite $2m$-regular graph. In this case $[F_m:H] = |\widehat{V}|$.
    \item \label{prop_4} $H$ is normal in $F_m$ if and only if $\widehat{\Delta}_H$ is a $2m$-regular graph and any vertex $v$ of $\widehat{V}$ can be considered as the root vertex of $\widehat{\Delta}_H$.
    \item \label{lem-euler-path} For any edge $e \in \widehat{E}$ there exists a reduced path $p$ in $\widehat{\Delta}_H$ such that it travels $e$ only once and $o(p)=t(p)=v_0$. 
    \item \label{essential_A_delta} The DFA $\A_{\widehat{\Delta}_H}$ is essential.
\end{enumerate}\end{lem}
\begin{proof}
For the proofs of above properties (\ref{prop_1})-(\ref{prop_4}) we refer the reader to Propositions 3.8 and 8.3, and Theorem 8.14 of \cite{MR1882114kapovich2002}. Property (\ref{lem-euler-path}) follows immediately from the definition of the extended core graph $\widehat{\Delta}_H$. Property (\ref{essential_A_delta}) follows from Lemma 3.9 of \cite{MR1882114kapovich2002}. 
\end{proof}

\begin{conv}
From now on, we will assume that $H$ is a non-trivial finitely generated subgroup of a free group $F_m=F(A)$, where $A=\{a_1, a_2, \ldots, a_m\}$. $\Sigma = A \cup A^{-1}$. By $v_0$ we always mean the root vertex of the (extended) core of $H$. By $\deg(v_0)$ we denote the in (=out) degree of $v_0$ in the extended core graph $\widehat{\Delta}_H$.
\end{conv} 

\section{The Nielsen system of generators}\label{S:4_Niel}
Let $H$ be a subgroup of $F_m$ generated by the set $S = \{w_i\}_{i=1}^k, 1 \leq k \leq \infty$ of freely reduced words over $\Sigma$. We further assume that $S$ is a Nielsen basis. Recall that a set $S$ of freely reduced words from $\Sigma^*$ has the Nielsen property if the following two conditions hold:
\begin{enumerate}
    \item \label{nielsen1} If $u, v \in S\cup S^{-1}$ and $u \neq v^{-1}$ then 
    $$|u\cdot v|_{\Sigma} \geq |u|_{\Sigma}, |u\cdot v|_{\Sigma} \geq |v|_{\Sigma}.$$
    \item \label{nielsen2} If $u, v, w \in S\cup S^{-1}$ and $u \neq w^{-1}$, $v \neq w^{-1}$ then 
    $$|u\cdot v\cdot w|_{\Sigma} > |u|_{\Sigma} + |w|_{\Sigma} -|v|_{\Sigma},$$
    where by $|w|_{\Sigma}$ we mean the length of the reduced word $w$ over $\Sigma$. 
\end{enumerate}
Condition (\ref{nielsen1}) means that not more than a half of $u$ and not more than a half of $v$ freely cancels in the product $u\cdot v$. Condition (\ref{nielsen2}) means that assuming (\ref{nielsen1}) after free cancellation in the product $u\cdot v\cdot w$ at least one letter of $v$ will remain un-cancelled. See \cite{magnus2004combinatorial}. From (\ref{nielsen1}) and (\ref{nielsen2}) it follows that $S$ is a free basis of the subgroup $\langle S\rangle$ of $F_m.$

Nielsen was the first who proved that every non trivial subgroup of $F_m$ has a set of generators with properties (\ref{nielsen1}) and (\ref{nielsen2}). His argument was quite involved. A simpler proof is given in \cite{magnus2004combinatorial}. In Theorem 3.4 of \cite{magnus2004combinatorial}, it is shown that any minimal Schreier system of generators of subgroup $H$ of $F_m$ has the Nielsen property. See also Proposition 6.7 in \cite{MR1882114kapovich2002} or \cite{MR1812024lyndon2001}.

Let us recall briefly how to get a Nielsen system geometrically. Let $H$ be a subgroup of $F_m$ and $\Gamma$ be the corresponding Schreier graph. Recall that labels of edges of $\Gamma$ belong to the set $A = \{a_1,\cdots,a_m\}$. Let $T$ be the spanning tree in $\Gamma$. The set of vertices of $T$ is same as the set of vertices $V$ of $\Gamma$. The tree $T$ is obtained from $\Gamma$ by deletion of some edges. Let $E'$ be the set of deleted edges. With each edge $e \in E'$ we associate an element $w_e$ of $H$ which is the word $\widehat{\mu}(p)$, where $p$ is the unique path in $\Gamma$ as described in Figure (\ref{fig:schreier_generator}).
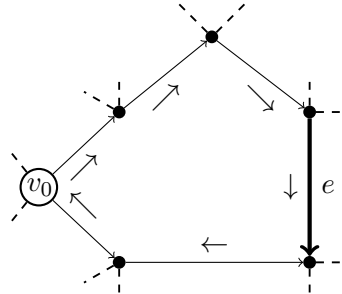
\begin{figure}[!htb]
    \centering
    \begin{tikzpicture}[scale=1]
    \tikzstyle{knode}=[circle,draw=black,thick,text
width = 1.5 pt,align=center,inner sep=1pt,fill]
\tikzstyle{rnode}=[circle,draw=black,thick,text
width = 9 pt,align=center,inner sep=1pt]
\node (q0) at (-2.8,0) [rnode]{$v_0$} ;
\node (q3) at (-1.7320,-1) [knode] {};
\node (q4) at (-1.7320,1) [knode] {};
\node (q5) at (-.5,2) [knode] {};
\node (q6) at (.8,1) [knode] {};
\node (q7) at (.8,-1) [knode] {};
\draw[->] (q0) to node[below] {$\nearrow$} (q4);
\draw[->] (q4) to node[below] {$\nearrow$} (q5);
\draw[->] (q5) to node[below] {$\searrow$} (q6);
\draw[->] (q0) to node[above] {$\nwarrow$} (q3);
\draw[->] (q3) to node[above] {$\leftarrow$}(q7);
\draw[->,ultra thick] (q6) to node[left] {$\downarrow$}node[right] {$e$} (q7);
\draw[-,dashed] (q5) edge (-1,2.5)
        (q5) edge (0,2.5)
        ;
\draw[-,dashed] (q0) edge (-3.2,-0.5)
         (q0) edge (-3.2,0.5)
         ;
\draw[-,dashed] (q3) edge (-2.2,-2.2*0.5773)
       (q3) edge (-1.7320,-1.5)
       ;
\draw[-,dashed] (q4) edge (-2.2,2.2*0.5773)
       (q4) edge (-1.7320,1.5)
       ;
\draw[-,dashed] (q6) edge (.8,1.5)
       (q6) edge (1.3,1)
       ;
\draw[-,dashed] (q7) edge (.8,-1.5)
       (q7) edge (1.3,-1)
       ;
    \end{tikzpicture}
    \caption{A Schreier generator}
    \label{fig:schreier_generator}
\end{figure}
That means, we connect the root $v_0$ with $o(e)$ by a reduced path that goes through the edges in $T$, then the path $p$ goes along the edge $e$ and after that goes from $t(e)$ to $v_0$ moving through the edges of $T$ in opposite direction (thus the return path $t(e) \longrightarrow v_0$ goes in fact through the edges in the extended Schreier graph $\widehat{\Gamma}$).

It is obvious that $w_e$ belongs to $H$. It is not obvious but the result of Schreier is that, the set $S = \{w_e\}_{e \in E'}$ is a free basis of $H$. Such a basis is called Schreier basis. The choice of a spanning tree is not canonical and usually there are a plenty of such choices (hence plenty of choices for Schreier system of generators). Some of the choices of $T$ are better than others. A spanning tree $T$ is geodesic (or minimal) with respect to the root $v_0$ if for any vertex $v \in V$ the combinatorial distance from $v$ to $v_0$ in $\Gamma$ is the same as the distance from $v$ to $v_0$ in $T$. Here we assume that we convert both graphs $\Gamma$ and $T$ into non-oriented graphs (by forgetting the direction of each edge). In this case the combinatorial distance (i.e. the number of edges in the closest path connecting two vertices) is the metric.

It is known that a spanning tree in the connected locally finite graph always exists and there is an effective procedure to find such a tree if the graph itself is defined in an effective way. It is well-known that for any spanning tree the corresponding Schreier system $S = \{w_e\}_{e\in E'}$ satisfies the Nielsen properties (\ref{nielsen1}) and (\ref{nielsen2}). It is straightforward that $H$ is a finite index subgroup in $F_m$ if and only if the core $\Delta_H$ coincides with the Schreier graph $\Gamma$ of $H$ (or the extended core $\widehat{\Delta}_H$ coincides with the extended Schreier graph $\widehat{\Gamma}_H$ of $H$). So if $[F_m:H] < \infty$, then we can have geodesic spanning trees $T_{\Delta_H}, T_{\Gamma}$ of $\Delta_H$ and $\Gamma$, respectively that coincide (i.e. $T_{\Delta_H} = T_{\Gamma}$). If $[F_m:H] = \infty$, then we can have $T_{\Delta_H} \subset T_{\Gamma}$. In this case, observe that the subtree $T_{\Gamma}\backslash T_{\Delta_H}$ is disconnected and the corresponding connected components are the hanging branches of $\Gamma$. Therefore, similarly to the finite index case, if $[F_m:H] = \infty$, then in order to find the Nielsen generating set it is sufficient to consider the $T_{\Delta_H}$. Let $E_{\Delta_H}' = E_{\Delta_H} \backslash E(T_{\Delta_H})$ be the set of deleted edges of $\Delta_H$. Then $S = S',$ where $S' = \{w_e\}_{e\in E_{\Delta_H}'}$. The highlighted geodesic spanning tree $T_{\Delta_H}$ in $\Delta_H$ of Figure (\ref{fig:folded}) can be used to find a Nielsen generating set $\{aba^{-1}, aca^{-1}\}$ of $H$.

\section{The construction of $\D_H$ and $\widehat{\D}_H$, their properties and consequences}\label{S:5_Constr}
The first goal of this section is to construct the minimal DFA $\D_H$ that recognizes the language $L_H$ of reduced words of a finitely generated subgroup $H \leq F_m.$ We shall approach this construction by defining an unambiguous automaton $\A_H$ as the (Cartesian) product of two automata. Then we obtain the minimal DFA $\D_H$ as the essential part of $\A_H$. At the end we obtain the multi-initial state automaton $\widehat{\D}_H$ by replacing the initial state of $\D_H$ and provide the complete description  when $\widehat{\D}_H$ is ergodic. In Theorem \ref{thm-on-computing-entropy}, we utilize the ergodicity property of $\widehat{\D}_H$ to obtain an entropy formula for $L_H$.

Recall that a finite automaton $\A$ is {essential} if in its Moore diagram $G_{\A}$ every vertex (hence, also every edge) belongs to some path connecting an initial state to a final state, i.e. to an admissible path. The following lemma will be used later.
\begin{lem}
\label{lem-on-essential-part}
Every finite automaton $\A$ has an essential part. If $\A$ is an unambiguous finite automaton, then it has only one essential part, called \emph{the} essential part of $\A$. 
\end{lem}
\begin{proof}
    Define $\A'$ to be the subautomaton of $\A$ such that $G_{\A'}$ is the union of all admissible paths of $G_{\A}$. Then, clearly, $L(\A')=L(\A)$ and $\A'$ is essential, hence $\A'$ is an essential part of $\A$.
    
    If $\A$ is unambiguous, then no proper subautomaton of $\A'$ generates the language $L(\A)$.
Also, if a vertex in $G_{\A}$ does not belong to some admissible path in $G_{\A}$, then, by definition, it will not belong to any essential subautomaton of $\A$. Thus, if $\A$ is unambiguous, then $\A'$ is the only essential part of $\A$.
\end{proof}
\subsection{The automaton $\A_H$}\label{S:5.1_A_H}
One of the important closure properties of regular languages is that the intersection of two regular languages is regular. See Theorem 3.3 on page 59 of \cite{MR645539Hopcroft1979}. A DFA recognizing the intersection of two regular languages can be constructed as follows. Let $L_1, L_2  \subseteq \Sigma^*$ be regular languages. Also, let $\A_1 = (Q_1,\Sigma,\delta_1,\{q_0\},\F_1)$ and $\A_2 = (Q_2,\Sigma,\delta_2,\{q_0'\},\F_2)$ be two DFA such that $L_1 = L(\A_1)$ and $L_2 = L(\A_2).$ Define the product $\A_1\times \A_2$ of $\A_1$ and $\A_2$ as follows.  $$ \A_1\times \A_2 = \left(Q_1\times Q_2, \Sigma, \delta, \{(q_0,q_0')\}, \F_1\times \F_2\right)$$ such that for all $q \in Q_1, q' \in Q_2$ and $x \in \Sigma$, we define 
    $$\delta((q,q'),x) = \left(\delta_1(q,x),\delta_2(q',x)\right) \textrm{  if  } \delta_1(q,x) \neq \emptyset, \delta_2(q',x)\neq \emptyset, \textrm{  and }$$
    $$\delta((q,q'),x) = \emptyset, \textrm{ otherwise. }$$ 
Then $\A_1\times \A_2$ is a DFA such that  $L(\A_1\times \A_2) = L_1\cap L_2$.

To construct a DFA $\A_{H}$ that recognizes the language $L_H$ of reduced words of $H$, we shall consider first the automaton $\A_1=\A_{\widehat{\Delta}_H} = \left(\widehat{V}, \Sigma, \delta_{\widehat{\Delta}_H}, \{v_0\}, \{v_0\}\right)$ discussed in Section 3. Recall that the Moore diagram of $\A_{\widehat{\Delta}_H}$ is the extended core $\widehat{\Delta}_H.$

We take a second automaton $\A_2 = \A_{F_m} = \left(Q_{F_m},\Sigma, \delta_{F_m}, \{q_0\}, \F_{F_m}\right),$ where the set of states and of final states are both equal to $$Q_{F_m} = \F_{F_m} = \{q_0\} \cup \left\{q_i^{\e}\mid a_i^{\e} \in \Sigma\right\}$$ and the transition function $\delta_{F_m}: Q_{F_m}\times\Sigma \rightarrow Q_{F_m}$ is 
\begin{eqnarray}
    \delta_{F_m}(q_0,a_j^{\e'}) = q_j^{\e'},& \textnormal{ for all } a_j^{\e'} \in \Sigma, \label{eqn:Fm_q0_transitions}\\
    \delta_{F_m}(q_i^{\e},a_j^{\e'}) = q_j^{\e'},& \textnormal{ if } a_i^{\e} \neq (a_j^{\e'})^{-1}. \label{eqn:Fm_qie_transitions}
\end{eqnarray}
See Figure \ref{fig:A_F2} for the Moore diagram of $\A_{F_2}$. We define $\A_H = \A_{\widehat{\Delta}_H} \times \A_{F_m} $. Namely,
$$\A_{H} = \A_{\widehat{\Delta}_H} \times \A_{F_m} = \left(Q_{H}, \Sigma, \delta_{H}, (v_0,q_0), \F_{H}\right),$$ where $$Q_{H} = \widehat{V}\times Q_{F_m} = \left\{(v,q) \mid v \in \widehat{V} \textrm{ and }q \in Q_{F_m} \right\},$$ 
$$\F_{H} = \{v_0\}\times \F_{F_m},$$ and 
\begin{align*}
    \delta_{H}((v_0,q_0),a_j^{\e'}) = \left(\delta_{\widehat{\Delta}_H}(v_0,a_j^{\e'}),\delta_{F_m}(q_0,a_j^{\e'})\right),
\end{align*}\text{ in case $\delta_{\widehat{\Delta}_H}(v_0,a_j^{\e'})\neq \emptyset$, $\delta_{F_m}(q_0,a_j^{\e'}) \neq \emptyset.$} Otherwise, $\delta_{H}((v_0,q_0),a_j^{\e'})=\emptyset.$

\begin{align*}
    \delta_{H}((v,q_i^{\e}),a_j^{\e'}) = \left(\delta_{\widehat{\Delta}_H}(v,a_j^{\e'}),\delta_{F_m}(q_i^{\e},a_j^{\e'})\right), 
\end{align*}
\text{in case $\delta_{\widehat{\Delta}_H}(v,a_j^{\e'})\neq \emptyset,\delta_{F_m}(q_i^{\e},a_j^{\e'}) \neq \emptyset.$} Otherwise, $\delta_{H}((v,q_i^{\e}),a_j^{\e'})=\emptyset.$

{\rem \label{rem:partial/complete_H} $\A_H$ inherits from $\A_{F_m}$ and $\A_{\widehat{\Delta}_H}$ the property of being deterministic and having only one initial state. In particular, $\A_H$ is an unambiguous automaton.} 

Recall that $L_H \subseteq \Sigma^*$ denotes the set of reduced words that represent elements of $H \leq F_m$.
\begin{prop}
$L(\A_{H})=L_H$.
\end{prop}

\begin{proof}
Recall that $L(\A_{H}) = L(\A_{\widehat{\Delta}_H})\bigcap L(\A_{F_m})$. Also $L_H \subseteq L(\A_{\widehat{\Delta}_H})\bigcap L(\A_{F_m})$, and $L(\A_{\widehat{\Delta}_H}) \setminus L_H$ consists of only non-reduced words, whereas $L(\A_{F_m})$ consists of only reduced words. Therefore, $L(\A_{H})=L_H$. 
\end{proof}
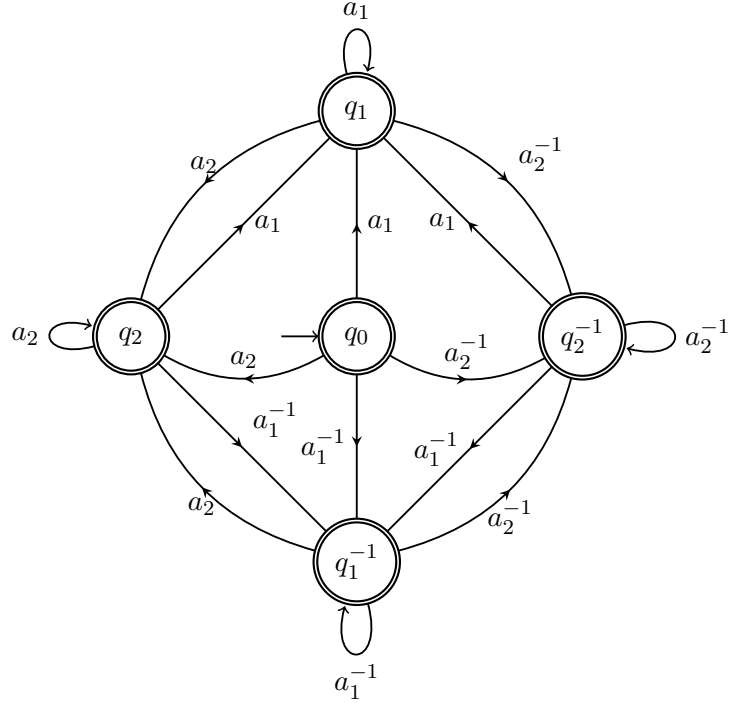
\begin{figure}[!htb]
    \centering
           \begin{tikzpicture}[scale=1,decoration={markings, mark= at position 0.5 with {\arrow{stealth}}}]
        \node[state, accepting, initial] (q0) {$q_0$};
        \node[state, accepting, above of=q0] (qa) {$q_1$};
        \node[state, accepting, below of=q0] (qa-) {$q_1^{-1}$};
        \node[state, accepting, left of=q0] (qb) {$q_2$};
        \node[state, accepting, right of=q0] (qb-) {$q_2^{-1}$};
       
        \draw  (q0) edge[above,postaction={decorate}] node[right]{$a_1$} (qa)
               (q0) edge[below,postaction={decorate}] node[left]{$a_1^{-1}$} (qa-)
               (q0) edge[bend left,postaction={decorate}] node[above]{$a_2$} (qb)
               (q0) edge[bend right,postaction={decorate}] node[above]{$a_2^{-1}$} (qb-)
               
               (qa) edge[loop above] node[above]{$a_1$} (qa)
               (qa) edge[bend right,postaction={decorate}] node[above]{$a_2$} (qb)
               (qa) edge[bend left,postaction={decorate}] node[above right]{$a_2^{-1}$} (qb-)

               (qa-) edge[loop below] node[below]{$a_1^{-1}$} (qa-)
               (qa-) edge[bend left,postaction={decorate}] node[below]{$a_2$} (qb)
               (qa-) edge[bend right,postaction={decorate}] node[below]{$a_2^{-1}$} (qb-)
               
               (qb) edge[loop left] node[left]{$a_2$} (qb)
               (qb) edge[above right,postaction={decorate}] node[right]{$a_1$} (qa)
               (qb) edge[below right,postaction={decorate}] node[above right]{$a_1^{-1}$}
               (qa-)
               
               (qb-) edge[loop right] node[right]{$a_2^{-1}$} (qb-)
               (qb-) edge[above left,postaction={decorate}] node[left]{$a_1$} (qa)
               (qb-) edge[below left,postaction={decorate}] node[left]{$a_1^{-1}$}
               (qa-);
    \end{tikzpicture}
    \caption{The Moore diagram of $\A_{F_2}$}
    \label{fig:A_F2}
\end{figure}
\subsection{Definition and main properties of $\D_H$}\label{S:5.2_D_H}
Define $Q_{\D_H}\subseteq \widehat V \times Q_{F_m}$ as 
$$Q_{\D_{H}} = \{(v_0, q_0)\}\cup\left\{(v,q_i^{\e}) \mid v \in \widehat{V} ~and~ \exists e \in \widehat{E} ~s.t.~ a_i^{\e} = \widehat{\mu}(e),t(e)=v\right\}.$$
In other words, $Q_{\D_{H}}$ is the set of accessible states of $\A_H$.

\begin{lem}\label{lem-00}\hfill  
\begin{enumerate}
\item If there is a path $p$  in $\A_{\widehat{\Delta}_H}$ with a reduced label $w$ such that $o(p)=v_0$ and $t(p)=v \in \widehat V$, and the suffix of $w$ is $a_i^{\e}$, then $(v,q_i^{\e}) \in Q_{\D_{H}}$.
\item If $(v,q_i^{\e}) \in Q_{\D_{H}}$, then there exists a reduced word $w_1a_i^{\e}w_2\in L_H$ such that $\delta_H(v_0, w_1a_i^{\e})= v $. 
\end{enumerate}
\end{lem}

\proof\hfill  
\begin{enumerate}
\item
 Let $w=a_{j_1}^{\epsilon_1} a_{j_2}^{\epsilon_2}  \ldots a_{j_{l}}^{\epsilon_{l}} \in L_H$ such that $a_{j_{l}}^{\epsilon_{l}} = a_i^{\e}$. Suppose that the path $v_0 \xrightarrow{a_{j_{1}}^{\epsilon_{1}}} v_1 \xrightarrow{a_{j_{2}}^{\epsilon_{2}}} \ldots \xrightarrow{a_{j_{l}}^{\epsilon_{l}}} v_l=v $ in $\A_{\widehat{\Delta}_H}$ with label $w$ and starting at $v_0$ terminates at $v \in \widehat{V}$.
Then, by definition of $\A_H$, the  path
$$(v_0, q_0) \xrightarrow{a_{j_{1}}^{\epsilon_{1}}} (v_1, q_{j_{1}}^{\epsilon_{1}}) \xrightarrow{a_{j_{2}}^{\epsilon_{2}}} \ldots \xrightarrow{a_{j_{l}}^{\epsilon_{l}}} (v_{l}, q_{j_{l}}^{\epsilon_{l}})=(v,q_i^{\e})$$ 
is well defined in $\A_H$ and has the label $w$. Also, by definition of $Q_{\D_{H}}$, $(v_{l}, q_{j_{l}}^{\epsilon_{l}}) \in Q_{\D_{H}}$. 
\item
Now assume that $(v,q_i^{\e}) \in Q_{\D_{H}}$. Then, by definition, there exists $e \in \widehat E$ such that $t(e)=v$ and $\hat{\mu}(e)=a_i^{\e}$. Assume that $v'=o(e) \neq v_0$. Then, since by Lemma \ref{lem:properties_core}, the number of incoming edges for $v'$ is at least two, there exists $e' \in \widehat E$ such that $t(e')=v'$ and its label $a_{i'}^{\e'}$ is different from $a_i^{\e}$. Therefore, $(v',q_{i'}^{\e'}) \in Q_{\D_{H}}$. 

Note that, by Lemma \ref{lem-euler-path}, $\widehat{\Delta}_H$  has an admissible path $p=p'ep''$ with reduced label such that $t(p')=v'$ and $o(p'')=v$. As the discussion in the proof of part (1) shows, $p$ corresponds to a path $\bar p$ in $\A_H$ with the same label as $p$. Now, if the label of $p'$ is $w_1$ and the label of $p''$ is $w_2$, then the label of $\bar p$ is $w_1 a_{i}^{\e} w_2$. By part (1) of the lemma, the sub-path of $\bar{p}$ with label $w_1 a_{i}^{\e}$ will terminate at $v$.  Therefore, by part (1) of the lemma, $\delta_H(v_0, w_1a_i^{\e})= v $. 
\qed\end{enumerate}
\begin{cor}
\label{corollary-admissible-path}
$(v,q_i^{\e}) \in Q_{\D_{H}}$ if and only if there is an admissible path in $\A_H$ that contains $(v,q_i^{\e})$ and whose vertices belong to $Q_{\D_{H}}$.
\end{cor}
\begin{proof}
    The `only if' part follows immediately from part (2) of Lemma \ref{lem-00}. For the `if' part, note that, by definition, any state that is on some admissible path in $\A_H$ is accessible, hence belongs to $Q_{\D_{H}}$.
\end{proof}

Define  $\D_{H}$ to be the subautomaton of $\A_{H}$ induced by the states $Q_{\D_H}$.

\begin{prop}
\label{prop-D_H-is-essntial}
$\D_{H}$  is the essential part of $\A_H = \A_{\widehat{\Delta}_H} \times \A_{F_m}$. In particular, $L_H = L(\D_H)$.
\end{prop}
\begin{proof}
    The fact that $\D_{H}$ is essential follows directly from Corollary \ref{corollary-admissible-path}. Also, from Lemma \ref{lem-00} and the fact that the set of reduced words in $L(\A_{\widehat{\Delta}_H})$ coincides with $L_H$ it follows that $L_H = L(\D_H)$. Combining this with Lemma \ref{lem-on-essential-part} and the fact that $\A_H$ is unambiguous, we get that $\D_{H}$ is the essential part of $\A_H$.
\end{proof}

{\rem \label{rem:partial/complete_H2} $\D_H$ inherits from $\A_{H}$  the properties of being deterministic and unambiguous. }

\noindent
\textbf{Automaton presentation of $\D_H$.} For expository reasons, in the sequel we replace the notation $q_i^{\epsilon}$ for a state of $\mathcal{D}_H$ by $a_i^{\epsilon}$, and denote the initial state $(v_0, q_0)$ simply by $q^*$. Thus $\D_H$ will have the following presentation: 
$$\D_{H} = \left(Q_{\D_{H}}, \Sigma, \delta_{\D_{H}}, \{q^*\}, \F_{\D_{H}}\right),$$ where
$$Q_{\D_{H}} = \{q^*\}\\\cup\left\{(v,a_i^{\e}) \mid v \in \widehat{V}, a_i^{\e} \in \Sigma, ~and~ \exists e \in \widehat{E} ~s.t.~ a_i^{\e} = \widehat{\mu}(e),t(e)=v\right\},$$
$$\F_{\D_{H}} = \{q^*\}\cup\left\{(v_0,a_i^{\e}) \mid \exists e \in \widehat{E} ~s.t.~ a_i^{\e} = \widehat{\mu}(e),t(e)=v_0\right\},$$ 
$$\delta_{\D_{H}}\left(q^*,a_i^{\e}\right) = (v_0a_i^{\e},a_i^{\e}), \textrm{ for all } a_i^{\e} \in \Sigma,$$
$$\delta_{\D_{H}}\left((v,a_i^{\e}),a_j^{\e'}\right) = \left(va_j^{\e'},a_j^{\e'}\right), \textrm{ if } a_i^{\e} \neq (a_j^{\e'})^{-1}.$$

\subsubsection{Minimality of $\D_{H}$}\label{S:5.2.1_Min_D_H}
The Myhill-Nerode Theorem (see Theorem \ref{NERODE:Theorem}) suggests the existence of a minimal DFA (which is unique up to isomorphism). 
For any given DFA $\A$, in order to construct the minimal DFA $\A'$ such that $L(\A) = L(\A')$, we need to define an equivalence relation $\equiv$ among states of $\A$.\\
We write $u \equiv v$ if the following holds: for all $w \in \Sigma^*$, $\delta(u, w)$ is a final state if and only if $\delta(v,w)$ is a final state. Notice that the relations $R_L$ (see \eqref{eqn:relation_R_L}) and $\equiv$ express exactly the same idea, i.e. 
\begin{equation}
     w R_L w'   \iff   \delta(q_0,w) \equiv \delta(q_0,w'). \label{eqn:fact}
\end{equation}
The states $u$ and $v$ of $\A$ are {\it equivalent} if $u \equiv v$. When the states $u$ and $v$ are not equivalent, then we say that they are {\it distinguishable}. That is, there exists at least one state $w$ such that one of $\delta(u, w)$ and $\delta(v, w)$ is an accepting state and the other is not.
\begin{thm}
$\D_{H}$ is a minimal DFA.
\end{thm}
\proof
In order to show that the DFA $\D_H$ is minimal, according to the Myhill-Nerode Theorem, it is enough to show that any two distinct states $(u,a_i^{\e})$ and $(v,a_j^{\nu})$ of $\D_H$ are distinguishable. 
We prove this by contradiction.

Assume that $(u,a_i^{\e}) \neq (v,a_j^{\nu})$ are  $R_L$-equivalent states. Then $\forall~w \in \Sigma^*$, $\delta_{\D_H}\left((u,a_i^{\e}), w\right)$ is a final state if and only if $\delta_{\D_H}\left((v,a_j^{\nu}), w\right)$ is a final state. Let us fix one such $w$ (the existence of $w$ follows from $\D_H$ being essential.) Since $\D_H$ is an essential automaton, there are $w_1$ and $w_2$ in $\Sigma^*$ such that $\delta_{\D_H}\left(q^*, w_1 \right) = (u,a_i^{\e})$ and $\delta_{\D_H}\left(q^*, w_2\right) = (v,a_i^{\nu})$. Therefore, $w_1w, w_2w \in L(\D_H) = L_H$, the suffix of $w_1$ is $a_i^{\e}$, the suffix of $w_2$ is $a_j^{\nu}$.  Then $w_1(w_2)^{-1} = (w_1w)(w_2w)^{-1} \in H$, which implies that $Hw_1=Hw_2$, and hence $u=v$. Therefore, it must be that $a_i^{\e} \neq a_j^{\nu}$, which, in particular, implies that $ w_1(w_2)^{-1} \in L_H$. As the states $(u,a_i^{\e})$ and $(v,a_j^{\nu})$ are equivalent, under the assumption that $(u,a_i^{\e})$ and $(v,a_j^{\nu})$ are $R_L$-equivalent, $w_1(w_2)^{-1} \in L_H$ implies $w_2(w_2)^{-1} \in L_H$, which is not true -- a contradiction. 
\qed

\subsection{Definition of $\widehat{\D}_H$}\label{S:5.3_hatD_H}
In this subsection we introduce the automaton $\widehat{\D}_H$, which is obtained from ${\D_H}$ by simply removing its initial state $q^*$ and proclaiming the final states $\F_{{\D_H}}$ of ${\D_H}$ at the same time the initial and final states of $\widehat{\D}_H$. Thus the automaton $\widehat{\D}_H$ would have the following presentation: $$\widehat{\D}_H = \left(Q_{\widehat{\D}_H}, \Sigma, \delta_{\widehat{\D}_H}, \I_{\widehat{\D}_H}, \F_{\widehat{\D}_H}\right),$$
where 
\begin{eqnarray}
    Q_{\widehat{\D}_H} = \left\{(v,a_i^{\e}) \mid v \in \widehat{V}, a_i^{\e} \in \Sigma, ~and~ \exists e \in \widehat{E} ~s.t.~ a_i^{\e} = \widehat{\mu}(e),t(e)=v\right\}, \nonumber \\
    \I_{\widehat{\D}_H} = \F_{\widehat{\D}_H} = \left\{(v_0,a_i^{\e})\mid a_i^{\e} = \widehat{\mu}(e),t(e)=v_0\right\}, \label{eqn:^D_H_initial_set}\\
    \delta_{\widehat{\D}_H}\left((v,a_i^{\e}),a_j^{\e'}) = (va_j^{\e'},a_j^{\e'}\right), \textrm{ if } a_i^{\e} \neq (a_j^{\e'})^{-1}. \label{eqn:^D_H_transition}
\end{eqnarray}
The main motivation behind consideration of the automaton $\widehat{\D}_H$ is that it will define the language $L_H$ and be ergodic whenever $L_H$ is irreducible, and at the same time will provide a convenient way of computing the entropy of $L_H$. All these will be revealed in detail in the next section. See, for example, Theorems \ref{thm-ergodicity}, \ref{thm-on-computing-entropy} and Proposition \ref{prop-when-L_H-is-irreducibl}.

To describe the main properties of $\widehat{\D}_H$, first we need the following definition.
\begin{defi}
$H\leq F_m$ is said to be \emph{conjugacy reduced} if there does not exist $a\in \Sigma$ such that every word from $L_H$ is of the form $a^{-1} w a$. In other words, $H$ is conjugacy reduced if there exist two words from $L_H$ with different last letters.
\end{defi}

\begin{prop}
\label{prop-1}
For $H\neq 1$, $\widehat{\D}_H$ has exactly one initial state if and only if $H$ is not conjugacy reduced. Moreover, in such case this state is isolated in the sense that there is no outgoing edge from it in the Moore diagram.
\end{prop}
\begin{proof}
    By definition of $\widehat{\D}_H$, it has only one initial state if and only if there is only one incoming edge for the root vertex of the extended core $\widehat{\Delta}_H$. The last property for $\widehat{\Delta}_H$ is equivalent to saying that all words from $L_H$ end with the same later.
    
    Finally, if the words in $L_H$ start with $a^{-1}$ and end with $a$ for some $a \in \Sigma$, then the initial state in $\widehat{\D}_H$ is $(v_0, a)$, which does not have an outgoing edge, because edges from $(v_0, a)$ with label $a^{-1}$ are not permitted by the definition of $\widehat{\D}_H$. Thus the proposition is proved.
\end{proof}

\begin{lem}
\label{lem-essentiality-of-hadD}
If $H\leq F_m$ is conjugacy reduced, then $\widehat{\D}_H$ is an essential automaton and $L_{\widehat{\D}_H} = L_H$.
\end{lem}
\begin{proof}
  By Proposition \ref{prop-1}, $H$ is conjugacy reduced means that $ \I_{\widehat{\D}_H}$ contains at least two elements $(v_0, a)$ and $(v_0, b)$ for $a\neq b \in \Sigma$. This means that whichever state one can move from $q^*$ in $\D_H$ is possible to move from at least one state from  $\I_{\widehat{\D}_H}$ (in fact, from precisely $|\I_{\widehat{\D}_H}|-1$ states.) Therefore, the fact that $\widehat{\D}_H$  is essential follows from $\mathcal{D}_H$ being essential, see Proposition \ref{prop-D_H-is-essntial}. 
  
  On the other hand, every state that is possible to attain from the states $\I_{\widehat{\D}_H}$ in $\widehat{\D}_H$ is possible to attain from $q^*$ in ${\D_H}$. Therefore,  $L_{\widehat{\D}_H} = L_{\mathcal{D}_H}= L_H$.
\end{proof}

\begin{prop}
\label{Prop:^D_H_equal_H} If $H\leq F_m$ is  conjugacy reduced, then $\widehat{\D}_H$ is deterministic and has homogeneous ambiguity $\deg(v_0) - 1$.
\end{prop} 
\begin{proof} 
The fact that $\widehat{\D}_H$  is deterministic follows straightforwardly from the definition of $\widehat{\D}_H$. The fact that $\widehat{\D}_H$ has homogeneous ambiguity $\deg(v_0) - 1$ follows from the observations that $\D_H$ is a DFA  (hence, is unambiguous) and the states that can be attained from $q^*$ in $\D_H$ are precisely the states that can be attained from precisely $|\I_{\widehat{\D}_H}|-1$ initial states of $\widehat{\D}_H$ (note that except the initial states, the states of $\D_H$ and $\widehat{\D}_H$ coincide). Therefore, $\widehat{\D}_H$ has homogeneous ambiguity $|\I_{\widehat{\D}_H}|-1=deg(v_0)-1$.\end{proof}

\subsection{Ergodicity of $\widehat{\mathcal{D}}_H$}\label{S:5.4_ergo_hatD_H}
The main goal of this subsection is to give a complete description of when $\widehat{\D}_H$ is ergodic for finitely generated subgroups $H$ of $F_m$. In Subsection \ref{section-on-entropy} we discuss some of its consequences that are connected with the entropy of $L_H$. Everywhere in this section we assume that $H$ is a non-trivial finitely generated subgroup of the free group $F_m=F(A)$, $\Sigma=A \cup A^{-1}$.

A finite state automaton is said to be \emph{ergodic} if its Moore diagram, regarded as a directed graph, is strongly connected. That is for each ordered pair of vertices of the Moore diagram there is a path from the first to the second vertex. An alternative characterization of ergodicity of a finite state automaton is that the formal language $L \subseteq \Sigma^*$ that it defines is \emph{irreducible}, which means that for each $w_1, w_2 \in L$ there exists $w \in \Sigma^*$ such that $w_1w w_2 \in L$. See \cite{MR2022842ceccherini2003}. Yet another characterization of ergodicity is that the adjacency matrix $M$ of the Moore diagram of the automaton is \emph{irreducible}: if $M$ is of size $m\times m$, then for each $1\leq i , j \leq m$, there exists $n \in \mathbb{N}$ such that the $(i, j)$-th coefficient of $M^n$ is strictly positive. The latest property of ergodic automata of bounded ambiguity can be used to compute the entropy of the corresponding language, which we discuss later in Subsection \ref{section-on-entropy}. 

\begin{prop}
\label{prop-2}
    $1 \neq H  \leq F_m$ is conjugacy reduced and cyclic if and only if there exist $a \neq b \in \Sigma$ such that all words from $L_H$ either start with $a^{-1}$ and end with $b$ or start with $b^{-1}$ and end with $a$. Moreover, in such case $\widehat{\D}_H$ has exactly two initial states, $(v_0, a)$ and $(v_0, b)$, which belong to two different connected components in the Moore diagram of $\widehat{\D}_H$. 
\end{prop}
\begin{proof}
First, assume that $L_H$ satisfies to the property from the statement. Then it follows from Proposition \ref{prop-1} that $H$ is conjugacy reduced. Now, by contradiction assume that $H$ is not cyclic. Then, let $a^{-1}w_1b$ be the shortest word in $L_H$ ending with $b$ and let  $a^{-1}w_2b \in L_H$ be the shortest word ending with $b$ that is not contained in the maximal cyclic subgroup of $H$ containing  $a^{-1}w_1b \in H$. Since $H$ is cyclically reduced, $a \neq b^{-1}$. Then $1\neq red(b^{-1}w_2^{-1} w_1 b ) \in L_H $. However, since by the assumption no word in $L_H$ starts with $b^{-1}$ and ends with $b$, it must be that after the free cancellation of $b^{-1}w_2^{-1} w_1 b$ either the prefix or the suffix of $b^{-1}w_2^{-1} w_1 b$ cancels out. In the first case we end up with a word that is shorter than $a^{-1}w_1b$, in the second case we end up with a word that is shorter than $a^{-1}w_2b$. Therefore, we end up with a contradiction because of the minimality condition on the lengths of $a^{-1}w_1b$ and $a^{-1}w_2b$. Thus $H$ is cyclic.

Now assume that $H$ is cyclic and conjugacy reduced. Then there exists $w \in H$ that generates $H$. In other words, $L_H = \{ w^n \mid n\in \mathbb{Z} \}$. Since $H$ is conjugacy reduced, the prefix of $w$ is not the inverse of the suffix. Therefore, the prefix corresponds to $a^{-1}$, the suffix corresponds to $b$, and the ``only if" part of the first statement of the proposition follows as well.

Now let us show the second statement. The condition that all words of $L_H$ are of the form $(a^{-1} w b)^{\pm 1}$ implies that the root vertex $v_0$ of $\widehat{\Delta}_H$ has exactly two incoming edges with corresponding labels $a$ and $b$, and it has exactly two outgoing edges with the same labels. This means that $\mathcal{I}_{\widehat{\D}_H} = \mathcal{F}_{\widehat{\D}_H} = \{(v_0, a), (v_0, b) \}$. Moreover, since by definition of $\widehat{\D}_H$ there is no outgoing edge of $(v_0, a)$ with label $a^{-1}$ in the Moore diagram, we have that the label of any accepted path that starts at  $(v_0, a)$ starts with $b^{-1}$ and hence ends with $a^{-1}$, hence is a loop. The same statement in regard with $(v_0, b)$ is true as well. Therefore, the second assertion of the proposition follows as well.
\end{proof}

\begin{thm}
\label{thm-ergodicity}
Let $H$ be a finitely generated subgroup of $F_m$. Then
$\widehat{\D}_H$ is an ergodic automaton if and only if $H$ is conjugacy reduced and non cyclic.
\end{thm}
\begin{proof}
First, assume that $H$ is conjugacy reduced and non cyclic.

By Lemma \ref{lem-essentiality-of-hadD}, $\widehat{\D}_H$ is essential, which means that any state of $\widehat{\D}_H$  is on some admissible path. Therefore, since $\I_{\widehat{\D}_H}=\F_{\widehat{\D}_H}$, in order to show that $\widehat{\D}_H$ is strongly connected, it is enough to show that for any ordered pair of different initial (=final) stats of $\widehat{\D}_H$, there is a directed path connecting the first state to the second one. (Note that, by Proposition \ref{prop-1}, $\widehat{\D}_H$ has at least two different initial states.)

Let $(v_0, a), (v_0, b)$, $a \neq b \in \Sigma$, be an arbitrary pair of initial states. We want to show that there exists a path from $(v_0, a)$ to $(v_0, b)$, which is equivalent to the existence of a word $w \in L_H$ which does not start with $a^{-1}$ and ends with $b$, as in such case $\delta_{\widehat{\D}_H}((v_0, a), w)= (v_0, b)$.

The fact that $(v_0, b) \in \mathcal{I}_{\widehat{\D}_H}$ means that there exists $u \in L_H$ such that it ends with letter $b$. Since by our assumption $H$ is not cyclic, there exists $1\neq v \in L_H$ such that $\langle u \rangle \cap  \langle v \rangle = \{1\}$. Then, for large enough $n \in \mathbb{N}$, the word $\bar v: = red( u^{-n} v u^{n}) \in L_H$ starts with $b^{-1}$ and ends with $b$. Therefore,  $\delta_{\widehat{\D}_H}((v_0, a), \bar v)= (v_0, b)$. Thus we showed that if $H$ is conjugacy reduced and non cyclic, then $\widehat{\D}_H$ is ergodic.

Now assume that $H$ is not conjugacy reduced. Then, by Proposition \ref{prop-1}, $\widehat{\D}_H$ has exactly one initial state, which is isolated. Therefore, in such case $\widehat{\D}_H$ is not ergodic. 

Finally, assume that $H$ is conjugacy reduced but cyclic. Then, by Proposition \ref{prop-2}, $\widehat{\D}_H$ has exactly two initial states which are not connected one to the other. Therefore,  in such case $\widehat{\D}_H$ is not ergodic either. Thus the theorem is proved.
\end{proof}

Combining Theorem \ref{thm-ergodicity} with Propositions \ref{prop-1} and \ref{prop-2}, we get the following corollary.

\begin{cor}
If $|\mathcal{I}_{\widehat{\D}_H}|>2$, then $\widehat{\D}_H$ is ergodic.
\end{cor}

\begin{rem}
Besides the case when $H$ is cyclic and conjugacy reduced, the only other case when $|\mathcal{I}_{\widehat{\D}_H}|=2$ is when, for some $a \in \Sigma$, $L_H$ consists of words only, and inclusively, of the forms $awa, awa^{-1}, a^{-1}wa^{-1}$ or $a^{-1}wa$. The latest corresponds, for example, to the case when $H= \langle aba^{-1}, a^2 \rangle < F(a, b) = F_2$.
\end{rem}

\begin{prop}
\label{prop-when-L_H-is-irreducibl}
For $H\leq F_m$, $L_H$ is irreducible if and only if $H$ is conjugacy reduced and non cyclic.
\end{prop}
\begin{proof}
    The `if' part of the statement follows from Theorem \ref{thm-ergodicity}.
    
    Now assume that $H$ is not conjucagy reduced, then for some $a \in \Sigma$, elements of $L_H$ are (reduced) words of the form $a^{-1} w a$. Since concatenation of any such words is not reduced, it means that $L_H$ is not irreducible in such case.
    
    Finally, assume that $H = \langle w \rangle$ is cyclic. Then, $w, w^{-1} \in L_H$. However, there is no $v \in L_H$ such that $wvw^{-1}$ is also in $L_H$. Therefore, $L_H$ is not irreducible in this case either.
    Thus the proposition is proved.
\end{proof}
\subsection{Entropy of $L_H$}
\label{section-on-entropy}
Recall that for a formal language $L$ its \emph{entropy}, $h(L)$, which is a fundamental numerical invariant of a language, is defined as 
\begin{equation}
     h(L) =  \limsup_{n \rightarrow \infty} { \frac{\log|{B_n(L)}|}{n}}, \label{eqn:entropy_def}
\end{equation}
where $B_n(L)$ is the subset of $L$ of words of length $n$.

If $L$ is a language generated by an automaton $\A$, then by $h(\A)$ we denote $h(L(\A))$. If for $H\leq F_m$, then the entropy of $H$ is $h(H)= h(L_H)$.

One important feature of ergodicity of $\widehat{\D}_H$ is that the adjacency matrix $M_{\widehat{\D}_H}$ of $\widehat{\D}_H$ is irreducible and is of bounded ambiguity (as follows from Theorem \ref{thm-ergodicity} and Proposition \ref{Prop:^D_H_equal_H}). Therefore, one can apply the Perron-Frobenius theory to obtain the following theorem on entropy of $L_H$. (For discussion on Perron-Frobenius theory see Chapter 4 in \cite{Lind-Marcus}.)

\begin{thm}
\label{thm-on-computing-entropy}
If $H \leq F_m$ is a conjugacy reduced and non cyclic finitely generated group, then the entropy of $L_H$ is equal to $\log\lambda$, where $\lambda$ is the maximal eigenvalue of the adjacency matrix $M_{\widehat{\D}_H}$ of $\widehat{\D}_H$.
\end{thm}
\begin{rem}
If $H$ is cyclic, then its entropy is equal to $0$. If $H$ is not conjugacy reduced, then there is $g \in F_m$ such that $gHg^{-1} \leq F_m$ is conjugacy reduced, and $h(H)=h(gHg^{-1})$. 
\end{rem}

\begin{defi}[Base automaton]
Let $\A$ be a finite automaton and let $(V, E)$ be its underlying graph. Then \emph{the base automaton $\breve{\A}$} of $\A$ is the automaton with the same underlying graph $(V, E)$ and such that 
\begin{enumerate}
    \item all states of $\breve{\A}$ are at the same time initial and final states,
    \item all edges of its Moore diagram are labeled with different labels. To be more specific, we assume that each $e \in E$ is labeled by the letter $x_e$.
\end{enumerate}
\end{defi}

To prove Theorem \ref{thm-on-computing-entropy}, we need the following fact.

\begin{prop}
\label{prop-aux}
If $\A$ is an essential automaton with bounded ambiguity, then $h(\A)=h(\breve{\A})$.
\end{prop}

\begin{proof}[Proof of Theorem \ref{thm-on-computing-entropy}.] Since for conjugacy reduced and non cyclic finitely generated $H\leq F_m$, by Lemma \ref{lem-essentiality-of-hadD} and Proposition \ref{Prop:^D_H_equal_H}, $\widehat{\D}_H$ is essential and of bounded ambiguity, then, by Proposition \ref{prop-aux},  $h(\widehat{\D}_H) = h(\breve{\D}_H)$, where by $\breve{\D}_H$ we denote the base automaton of $\widehat{\D}_H$. Additionally, by Theorem \ref{thm-ergodicity}, $\widehat{\D}_H$ is an ergodic automaton, which implies that  $\breve{\D}_H$ is ergodic as well. Moreover, the adjacency matrices of $\widehat{\D}_H$ and $\breve{\D}_H$ coincide and are irreducible. Therefore, the proof of Theorem \ref{thm-on-computing-entropy} follows from Perron-Frobenius theory. See, for example, Theorem 4.4.4. in \cite{Lind-Marcus}.
\end{proof}

\noindent
\textbf{Proof of Proposition \ref{prop-aux}.}
Assuming that $\A$ is an essential automaton, for every vertex $v\in V$ in its underlying graph $(V, E)$, there exist paths that connect $v$ to an initial state and to a final state, respectively. For each $v$ let us fix a pair $(\check{v}, \hat{v})$ of such paths, where $\check{v}$ is a path that connects an initial state to $v$ and $\hat{v}$ is a path that connects $v$ to a final state. Also, for each path $p$ in $(V, E)$ let us denote by $p_-$ its origin and by $p_+$ is terminus, and by $\phi(p)$ let us denote its label in the Moore diagram of $\A$ and by $\psi(p)$ let us denote its label in the Moore diagram of $\Breve{\A}$. Note that for each path $p$, $\psi(p) \in L(\Breve{\A})$. Then we have the following map $\Lambda: L(\Breve{\A}) \rightarrow L(\A)$: for each path $p$ in $(V, E)$, define 
$$ \Lambda(\psi(p)) = \phi( \check{p}_- p \hat{p}_+).$$

\begin{lem}
\label{lem-51}
There exists $C\in \mathbb{N}$ such that for each path $p$ from $(V, E)$, $|p| \leq |\Lambda(p)| \leq |p|+C$.
\end{lem}
\begin{proof}
    The left inequality is obvious. For the right inequality, one can take $C=\max\{|\check v| \mid v \in V\} + \max\{|\hat v| \mid v \in V\}$.
\end{proof}

From Lemma \ref{lem-51}, for all $n \in \mathbb{N}$, we immediately get
\begin{align}
    \label{eq-51-1}
    \Lambda(B_n(L(\Breve{\A}))) \subseteq \bigcup_{i=0}^C B_{n+i}(L(\A)).
\end{align}

\begin{lem}
\label{lem-52}
If $\A$ has bounded ambiguity, then there exists $D>0$ such that for all $w \in L(\A)$, $|\Lambda^{-1}(w)| \leq D$.
\end{lem}
\begin{proof}
    Let $K>0$ be an upper bound of ambiguity of $\A$. Let $w \in L(\A)$. Also, let $(u, v) \in V \times V$. Then the number of paths $p$ in $(V, E)$ such that $p_-=u$, $p_+=v$, and  $\Lambda(\psi(p))=w$ is bounded from above by $K$. The number of pairs $(u, v) \in V \times V$ is equal to $|V|^2$. Therefore, one can take $D= K |V|^2$.
\end{proof}

\begin{lem}
\label{lem-final-5}
For all $n \in \mathbb{N}$, there exists $0 \leq c_n \leq C$ such that 
\begin{align}
\label{inq}
    |B_{n}(L(\A))| \leq  |B_n(L(\Breve{\A}))|   \leq (C+1)D|B_{n+c_n}(L(\A))|.
\end{align}
\end{lem}
\begin{proof}
    The left inequality of \eqref{inq} is trivial. For the right one, note that for all $n \in \mathbb{N}$, by Lemma \ref{lem-52} and by \eqref{eq-51-1}, we get
    $$D \sum_{i=0}^C  |B_{n+i}(L(\A))| \geq  |B_n(L(\Breve{\A}))|.$$
    Therefore, one can take $c_n$ in \eqref{inq} to be such that 
    $$|B_{n+c_n}(L(\A))| = \max\{|B_{n+i}(L(\A))| \mid 0\leq i \leq C \}.$$
\end{proof}

From Lemma \ref{lem-final-5} we immediately get that there exists a constant $C'>0$ and a sequence $\{c_n\}_{n=1}^{\infty}$, $0 \leq c_n \leq C'$, such that 

$$ \frac{\log  |B_{n}(L(\A))|}{n} \leq   \frac{\log  |B_n(L(\Breve{\A}))| }{n} \leq \frac{\log  C'|B_{n+c_n}(L(\A))|}{n} .$$

Proposition \ref{prop-aux} follows immediately from the last inequality.
\section{Computing cogrowth series of $H$}\label{sec:computing_growth}\label{S:6_cogrow}
With an arbitrary subgroup $H$ of $F_m$ one associates the \emph{growth function} $$\gamma_H(n) = |H_n|,$$ where $H_n$ is the set of elements in $H$ of length $n$ with respect to the basis $A$ of $F_m$ i.e. length of the element is the length of the reduced word from $\Sigma^*$, where $\Sigma = A \cup A^{-1}$ representing the element. Also, following \cite{MR599539Gri1980} we introduce the \emph{cogrowth series }
\begin{equation}
    H(z) = \sum_{n = 0}^{\infty} |H_n|z^n \label{gri_cogrowth}
\end{equation}
The upper limit 
\begin{equation}
    \alpha_H = \limsup_{n \rightarrow \infty} |H_n|^{1/n} \label{growth_rate}
\end{equation}
is called the \emph{growth rate} of $H$ with respect to the basis $A$ of $F_m$. The radius of convergence of the series (\ref{gri_cogrowth}) is $$R = \displaystyle\frac{1}{\alpha_H}$$
and $$R \geq \displaystyle \frac{1}{2m-1} \textrm{ as, for } H = F_m,~~ |H_n| = 2m(2m-1)^{n-1}  \textrm{ if } n \geq 1.$$

The cogrowth series (\ref{gri_cogrowth}) represents a function of complex variable $z \in \mathbb{C}$ analytic at disc around $z = 0$ of radius $R \geq \displaystyle \frac{1}{2m-1}$. The major questions are: Under what conditions $H(z)$ is rational, algebraic and belongs to distinguished class of analytic functions, like for instance the class of $D$-finite functions studied in \cite{MR3478442pak2016}.

Let us recall the second author's argument from \cite{MR599539Gri1980} on rationality of $H(z)$ (when $H$ is finitely generated) using a Nielsen system of generators of $H$.

\subsection{The Nielsen basis approach}\label{S:6.1_Niel_basis}
Let $H = \langle w_1,\cdots,w_k\rangle$ and let $\{w_i\}_{i=1}^k$ be a Nielsen system of generators. Given two reduced words $u, v \in \displaystyle\Sigma^*$ denote by $\beta(u,v)$ the number of $a$-symbols (by $a$-symbols we mean elements of $\Sigma$) which will be cancelled when reducing the product $u\cdot v$. For each $i, 1\leq i \leq k$ and $\e \in \{-1,1\}$ let $H_n^{i,\e}$ be the set of words of length $n$ in $H$ that can be presented by $S$-reduced product of generators from $S\cup S^{-1}, S = \{w_1,\cdots,w_k\}$ ending with $w_i^{\e}$. That is
$$u \in H_n^{i,\e} \Longleftrightarrow u = red(w_{i_1}^{\e_1}\cdots w_{i_l}^{\e_l}~w_{i}^{\e})$$
for some $i_1,\cdots,i_l,\e_1,\cdots,\e_l \in \{-1,1\}$ and in the last product none of the factors $w_{i_j}^{\e_j}$ is inverse to the previous or next factor. Let us recall the corollary of the Statement 3.6 of \cite{MR599539Gri1980}.
\begin{thm}
\label{cor_statement 3.2} If $H$ is a finitely generated subgroup of $F_m$, then $H(z)$ is rational.
\end{thm}
\proof
Let $H_i^{\e}(z) = \sum_{n=0}^{\infty}|H_n^{i,\e}|z^n$.
Then $$H(z) = 1 + \sum_{i,\e}H_i^{\e}(z).$$
The equation $(3.5)$ from \cite{MR599539Gri1980} shows that the functions $H_i^{\e}(z)$ satisfy the system of linear equations. 
\begin{equation}
    H_i^{\e}(z) = z^{|w_i^{\e}|} + \displaystyle \sum_{j,\e'} z^{|w_i^{\e}|-\beta(w_j^{\e'},w_i^{\e})} H_j^{\e'}(z) \label{gri_system}
\end{equation}
Taking the summation term to the left, we can rewrite the above system as follows
\begin{eqnarray}
 BY = Z
\end{eqnarray}
where $B$ is a $2k\times 2k$ matrix with $(i,\e),(j,\e')$-th entry
$$B\left((i,\e),(j,\e')\right) = \displaystyle \chi_i(j) - z^{|w_i^{\e}|-\beta(w_j^{\e'},w_i^{\e})},$$
where $$\chi_i(j) = \left\{ \begin{array}{cc}
    1  & \textnormal{ if } i = j \\
    0  & \textnormal{ otherwise}
\end{array} \right.$$
$Y$ and $Z$ are $2k\times 1$ column vectors whose $i$-th entries are
$$Y_i = H_i^{\e}(z),~~ Z_i = z^{|w_i^{\e}|},$$ respectively. Observe that, when $z \in \R$ and $|z| < 1$, the determinant of the matrix $B$ is non-zero. Hence the system (\ref{gri_system}) has a unique solution. Solving this system by standard methods (for instance using Cramer's rule) we get a solution. 
$$H_i^{\e}(z) = \displaystyle \frac{P_i^{\e}(z)}{Q_i^{\e}(z)}, i = 1,\cdots,k; \e \in \{-1,1\},$$
where $P_i^{\e}, Q_i^{\e}$ are polynomials and hence we get the rational expression for $H(z)$\\
$$H(z) = 1 + \displaystyle \sum_{i,\e} \frac{P_i^{\e}(z)}{Q_i^{\e}(z)}.$$ \qed
\subsection{The finite automata approach}\label{S:6.2_auto}
In the previous section we have seen the construction of the DFA $\D_H$. It is a very old observation, going back to Chomsky and Sch\"utzenberger and even to A. Kolmogorov in view of his theory of finite Markov chains, that the growth series $L_H(z)$ of the language $L(\D_H) = L_H$ which by Proposition \ref{prop-D_H-is-essntial} from the previous sections coincides  with $H(z)$ and are rational. Moreover, it can be computed using a standard method which often is called the transfer matrix method. See page 573 of \cite{MR2868112Stanley2012} or Section V.5 and in particular Proposition V.6 of \cite{MR2483235flajolet2009}. In this section, using the DFA $\D_H$ we shall compute the cogrowth series $H(z)$. At the end, we shall present the computations of growth using Nielsen set of generators.

Let $M$ be the adjacency matrix of the labelled directed graph $G$ with $t$ vertices. Assume that, for every vertex of the graph $G$, all outgoing edges carry distinct labels. Let $M^n$ be the $n$th power of $M$. It is well known that the $(i,j)$th entry of $M^n$ which we denote by $M^n(i,j)$ is just the number of paths of length $n$ from the $i$th vertex to the $j$th vertex of $G$. To evaluate the $M^n(i,j)$, we shall use the transfer matrix method.  This method uses linear algebra to analyze the behavior of the $M^n(i,j)$. Following \cite{MR2868112Stanley2012}, we define the  growth series (also known as generating function) of paths in $G$ from $i$ to $j$. 
$$ \gamma_{ij}(G,z) = \sum_{n \geq 0} M^n(i,j)z^n.$$
Observe that $\gamma_{ij}(G,z)$ is the $(i,j)$th entry of the matrix $$\displaystyle \sum_{n\geq 0} M^n z^n = \left(I - z M\right)^{-1}$$ where $I$ is the identity matrix of dimension $t\times t$. In order to compute the $(i,j)$th entry $\gamma_{ij}(G,z)$, we recall Theorem 4.7.2 of \cite{MR2868112Stanley2012}. 
\begin{thm}
\label{4.7.2} The growth series $\gamma_{ij}(G,z)$ is given by
\begin{equation}
    \gamma_{ij}(G,z) = \displaystyle \frac{(-1)^{i+j}\det\left( I - zM : j, i\right)}{\det\left(I - zM \right)}
\end{equation}
where $(B : j, i)$ denotes the minor obtained by removing the $j$th row and $i$th column of $B.$ Thus in particular $\gamma_{ij}(G,z)$ is a rational function of $z$ whose degree is strictly less than the multiplicity $n_0$ of $0$ as an eigenvalue of $M.$
\end{thm} 

Let $L_H = L(\D_H)$ be the language accepted by the DFA $\D_H$ constructed in the previous Section. Let $M_{\D_H}$ be the adjacency matrix of the Moore diagram of $\D_H$ and $L_H(z)$ be the growth of $L_H$, i.e. $$L_H(z) = \displaystyle \sum_{w \in L_H} z^{|w|}. $$
Recall that $\D_H$ has $|Q_{\D_H}| = 1+\displaystyle\sum_{v \in \widehat{V}}\deg(v)$ states. In our numeration of states of $\D_H$ we start with the initial state $q^*$ first and then the states $(v_k,a_i^{\e_i})$ where $k = 0,\cdots,\left(|\widehat{V}|-1\right), i \in \{1,\cdots,m\}$ and $\e_i \in \{1,-1\}$. Let $\initial$ be the standard unit row vector of dimension $1\times |Q_{\D_H}|$ and let the column vector $\final$ be the characteristic vector of the set of final states of dimension $|Q_{\D_H}|\times 1$. 
{\prop \label{pro:cogrow_from_auto} \begin{equation}
      H(z) = L_H(z) = \initial\cdot\left(I - z M_{\D_H}\right)^{-1}\cdot\final, \label{transfer}
\end{equation}
where $I$ is the identity matrix of order $|Q_{\D_H}|\times|Q_{\D_H}|$.}
\begin{proof} Recall that $L(\D_H) = L_H = $ set of reduced elements of $H$ which implies that $H(z) = L_H(z)$. For every $w \in L_H$ of length $n \geq 0$ we have a unique admissible path $p$ in the Moore diagram $G_{\D_H}$ of $\D_H$ such that $w = l(p)$. Therefore, we write the growth series $L_H(z)$ as
    \begin{eqnarray}
    H(z) = L_H(z) &=& \displaystyle \sum_{q \in \F_{\D_H}} \displaystyle \gamma_{q^*q}(G_{\D_H},z) \nonumber\\
              &=& \initial\cdot\left(I - z M_{\D_H}\right)^{-1}\cdot\final \nonumber.
    \end{eqnarray}
\end{proof}
{\rem \label{growth_rate=log_entropy} Observe that from (\ref{eqn:entropy_def}) and (\ref{growth_rate}) one deduces $$ \alpha_H = e^{h(L_H)}, $$ where $h(L_H)$ is the entropy of $L_H$.}
\section{Examples}\label{S:7_exa}
Let $H$ be a nontrivial finitely generated subgroup of $F_m$. In this section we shall compute the cogrowth $H(z)$ of $H$.
\subsection{Computations of cogrowth using DFA $\D_H$}\label{section7.1}
In each of the examples below, we consider DFA $\D_H$ that recognize language $L_H$ of reduced elements of $H$. Let $M_{\D_H}$ be the adjacency matrix of the Moore diagram of $\D_H$.
\begin{enumerate}
    \item Let $H$ be a finite index subgroup of $F_m$.
        \begin{enumerate}
            \item $H = \langle a^2, b, c, aba^{-1}, aca^{-1}\rangle$ is a subgroup of $F_3$.
            In this case, the core $\Delta_H$ is a complete graph. i.e. for each vertex there are three outgoing and three incoming edges labelled by $a,b,c$ and their inverses, respectively. Hence $\Delta_H$ coincides with the Schreier graph $\Gamma$ of $H$. The index of $H$ in $F_3$ $ = |\widehat{V}| = 2$. Moreover, any vertex of $\Delta_H$ can be considered as a root vertex, which implies that $H$ in fact is a normal subgroup. The diagram of core $\Delta_H$ and the automaton $\D_H$ are shown in the Figures (\ref{fig:core_normal1}) and (\ref{fig:auto_normal1}), respectively. The adjacency matrix $M_{\D_H}$ of $G_{\D_H}$ is given below
                 $$ M_{\D_H} = \left(
\begin{array}{ccccccccccccc}
 0 & 0 & 0 & 1 & 1 & 1 & 1 & 1 & 1 & 0 & 0 & 0 & 0 \\
 0 & 0 & 0 & 1 & 1 & 1 & 1 & 1 & 0 & 0 & 0 & 0 & 0 \\
 0 & 0 & 0 & 1 & 1 & 1 & 1 & 0 & 1 & 0 & 0 & 0 & 0 \\
 0 & 0 & 0 & 1 & 0 & 1 & 1 & 1 & 1 & 0 & 0 & 0 & 0 \\
 0 & 0 & 0 & 0 & 1 & 1 & 1 & 1 & 1 & 0 & 0 & 0 & 0 \\
 0 & 0 & 0 & 1 & 1 & 1 & 0 & 1 & 1 & 0 & 0 & 0 & 0 \\
 0 & 0 & 0 & 1 & 1 & 0 & 1 & 1 & 1 & 0 & 0 & 0 & 0 \\
 0 & 0 & 1 & 0 & 0 & 0 & 0 & 0 & 0 & 1 & 1 & 1 & 1 \\
 0 & 1 & 0 & 0 & 0 & 0 & 0 & 0 & 0 & 1 & 1 & 1 & 1 \\
 0 & 1 & 1 & 0 & 0 & 0 & 0 & 0 & 0 & 1 & 0 & 1 & 1 \\
 0 & 1 & 1 & 0 & 0 & 0 & 0 & 0 & 0 & 0 & 1 & 1 & 1 \\
 0 & 1 & 1 & 0 & 0 & 0 & 0 & 0 & 0 & 1 & 1 & 1 & 0 \\
 0 & 1 & 1 & 0 & 0 & 0 & 0 & 0 & 0 & 1 & 1 & 0 & 1 \\
\end{array}
\right)$$
Applying formula (\ref{transfer}) we get
                $$H(z) = \displaystyle \frac{(1+z)\left(1-4z+5 z^2\right)}{(1-5z)\left(1-2z+5z^2\right)}
                    \textnormal{ and  } \alpha_H = 5.$$
\begin{figure}[!htb]
       \begin{subfigure}[b]{0.35\textwidth}
        \centering
        \begin{tikzpicture}[scale=1,decoration={markings, mark= at position 0.5 with {\arrow{stealth}}}]
                 \node[state] (q0) {$v_0$};
                 \tikzstyle{knode}=[circle,draw=black,thick,text
                    width = 3 pt,align=center,inner sep=1pt]
                 \node[state, above of=q0] (q1) {$v_1$};
                 \draw   (q0) edge[ultra thick, bend left, left,postaction={decorate}] node{$a$} (q1)
                         (q1) edge[bend left, left,postaction={decorate}] node{$a$} (q0)
                         ;
                         \path (q1) edge [loop right, right] node[right]{$b$}  (q1);
                         \path (q1) edge [loop left, left] node[left]{$c$} (q1);
                        \path (q0) edge [loop right, right] node[right]{$b$}  (q0);
                        \path (q0) edge [loop left, left] node[left]{$c$} (q0);
                        \end{tikzpicture}
                 \caption{}
        \label{fig:core_normal1}
       \end{subfigure}
       \hfill
       \begin{subfigure}[b]{0.65\textwidth}
       \centering
       \includegraphics[width=\textwidth]{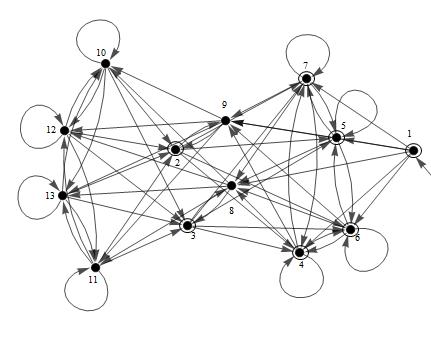}
       \caption{}
        \label{fig:auto_normal1}
       \end{subfigure}
       \caption{\eqref{fig:core_normal1} The core graph $\Delta_H$ of $H$. The highlighted spanning tree $T$ in $\Delta_H$ can be used to find a free basis $\{a^2, b, c, aba^{-1}, aca^{-1}\}$ of $H$. \eqref{fig:auto_normal1} is the Moore diagram of $\D_H$. The initial state has label $1$ and it is indicated by an incoming arrow. The final states of $\D_{H}$ have labels $1,2,3,4,5,6$ and $7$, respectively and they are indicated by the circles around the states.}
        \end{figure}
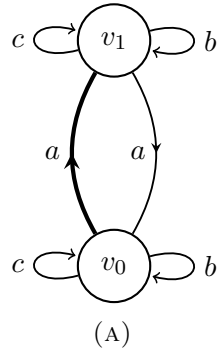
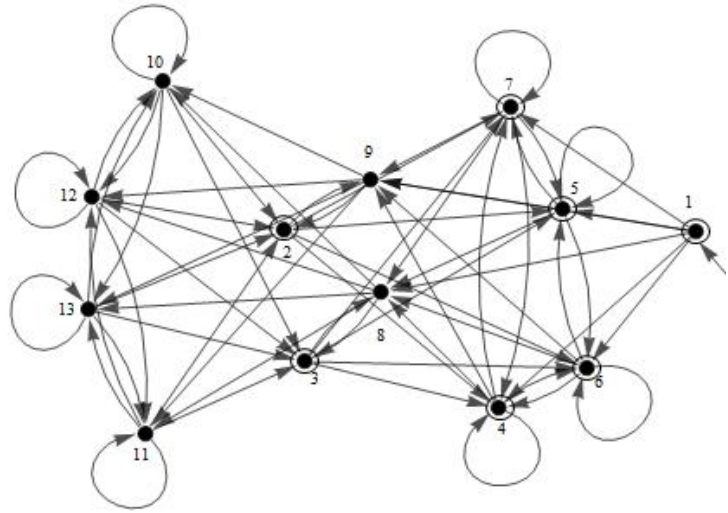
    \item $H = \langle a^3,ab,ab^{-1},a^{-1}ba\rangle$ is a subgroup of $F_2$.
        In this case, the core $\Delta_H$ coincides with the Schreier graph $\Gamma$ of $H$. The index of $H$ in $F_2 = |\widehat{V}| = 3.$ 
        Moreover, any vertex of $\Delta_H$ can not be considered as a root vertex, and hence $H$ is not a normal subgroup. See Figures \eqref{fig:core_index3} and \eqref{fig:auto_index3} for the core $\Delta_H$ and the diagram of the automaton $\D_H$, respectively. The adjacency matrix $M_{\D_H}$ is
              $$ M_{\D_H} = \left(
\begin{array}{ccccccccccccc}
 0 & 0 & 0 & 0 & 0 & 1 & 0 & 1 & 1 & 0 & 1 & 0 & 0 \\
 0 & 0 & 0 & 0 & 0 & 1 & 0 & 1 & 1 & 0 & 0 & 0 & 0 \\
 0 & 0 & 0 & 0 & 0 & 0 & 0 & 1 & 1 & 0 & 1 & 0 & 0 \\
 0 & 0 & 0 & 0 & 0 & 1 & 0 & 1 & 0 & 0 & 1 & 0 & 0 \\
 0 & 0 & 0 & 0 & 0 & 1 & 0 & 0 & 1 & 0 & 1 & 0 & 0 \\
 0 & 0 & 0 & 1 & 1 & 0 & 0 & 0 & 0 & 1 & 0 & 0 & 0 \\
 0 & 0 & 1 & 1 & 1 & 0 & 0 & 0 & 0 & 0 & 0 & 0 & 0 \\
 0 & 0 & 1 & 1 & 0 & 0 & 0 & 0 & 0 & 1 & 0 & 0 & 0 \\
 0 & 0 & 1 & 0 & 1 & 0 & 0 & 0 & 0 & 1 & 0 & 0 & 0 \\
 0 & 1 & 0 & 0 & 0 & 0 & 0 & 0 & 0 & 0 & 0 & 1 & 1 \\
 0 & 0 & 0 & 0 & 0 & 0 & 1 & 0 & 0 & 0 & 0 & 1 & 1 \\
 0 & 1 & 0 & 0 & 0 & 0 & 1 & 0 & 0 & 0 & 0 & 1 & 0 \\
 0 & 1 & 0 & 0 & 0 & 0 & 1 & 0 & 0 & 0 & 0 & 0 & 1 \\
\end{array}
\right)$$
 Applying formula (\ref{transfer}) we get
 $$ H(z) =  \displaystyle \frac{(1+z)(1-2z+5z^2-6z^3+9z^4)}{(1-3z) \left(1-z+3z^2\right) \left(1+3z+3z^2\right)} \textnormal{ and } \alpha_H = 3.$$
     \begin{figure}[!htb]
       \begin{subfigure}[b]{0.3\textwidth}
        \centering
        \begin{tikzpicture}[scale=1,,decoration={markings, mark= at position 0.5 with {\arrow{stealth}}}]
                 \node[state] (q0) {$v_0$};
                 \node[state, above of =q0] (q1) {$v_1$};
                 \node[state, right of=q0] (q2) {$v_2$};
                 \draw   (q0) edge[ultra thick, bend left, left,postaction={decorate}] node{$a$} (q1)
                         (q0) edge[bend right, right,postaction={decorate}] node{$b$} (q1)
                         (q1) edge[below, right,postaction={decorate}] node{$b$} (q0)
                         (q2) edge[ultra thick, below,postaction={decorate}] node{a} (q0)
                         (q1) edge[right, above,postaction={decorate}] node{$a$} (q2)
                         (q2) edge[loop above, above] node{$b$} (q2)
                         ;
                 \end{tikzpicture}
                 \caption{}
        \label{fig:core_index3}
       \end{subfigure}
       \hfill
       \begin{subfigure}[b]{.6\textwidth}
       \centering
       \includegraphics[width=\textwidth]{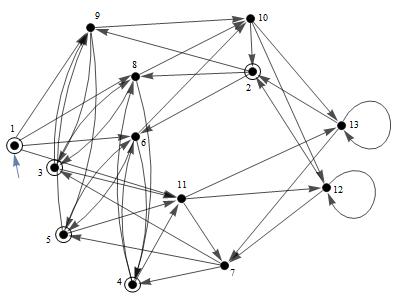}
       \caption{}
        \label{fig:auto_index3}
       \end{subfigure}
       \caption{\eqref{fig:core_index3} The core graph $\Delta_H$ of $H = \langle a^3,ab,ab^{-1},a^{-1}ba\rangle$. \eqref{fig:auto_index3} is the Moore diagram of $\D_H$.}
    \end{figure}
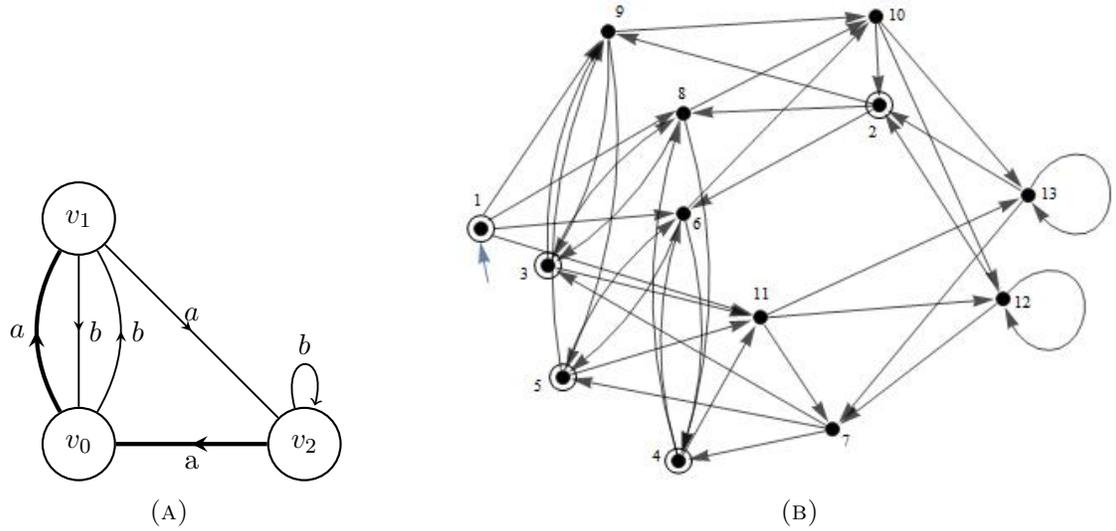
        \end{enumerate}
        \item Let $H$ be an infinite index subgroup of $F_m$.
        \begin{enumerate}
            \item $H = \langle aba^{-1}, aca^{-1} \rangle$ is a subgroup of $F_3$. 
        In this case, the core $\Delta_H$ is incomplete graph and it is a subgraph of $\Gamma.$ So $H$ is an infinite index subgroup of $F_3$. See Figures (\ref{fig:1core_inf_index1}) and (\ref{fig:auto_inf_index1}) for the core $\Delta_H$ and the diagram of the automaton $\D_H$, respectively. The $\deg(v_0) = 1.$ The adjacency matrix $M_{\D_H}$ of $G_{\D_H}$ is 
              $$ M_{\D_H} = \left(
\begin{array}{ccccccc}
 0 & 0 & 1 & 0 & 0 & 0 & 0 \\
 0 & 0 & 0 & 0 & 0 & 0 & 0 \\
 0 & 0 & 0 & 1 & 1 & 1 & 1 \\
 0 & 1 & 0 & 1 & 0 & 1 & 1 \\
 0 & 1 & 0 & 0 & 1 & 1 & 1 \\
 0 & 1 & 0 & 1 & 1 & 1 & 0 \\
 0 & 1 & 0 & 1 & 1 & 0 & 1 \\
\end{array}
\right)$$
       Applying formula (\ref{transfer}) we get
       $$ H(z) =  \displaystyle 1+ \frac{4z^3}{1-3z} \textnormal{ and } \alpha_H =  3.$$
       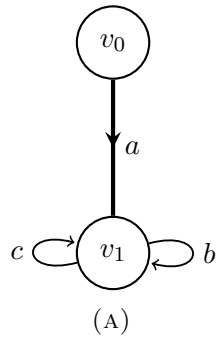
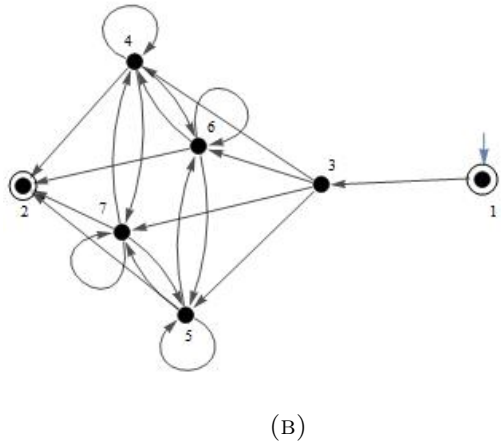
\begin{figure}[!htb]
\begin{subfigure}[b]{0.5\textwidth}
        \centering
         \begin{tikzpicture}[scale=1,decoration={markings, mark= at position 0.5 with {\arrow{stealth}}}]
        \tikzstyle{knode}=[circle,draw=black,thick,text
                width = 3 pt,align=center,inner sep=1pt,fill]
                \node (q1) at (0,-2.8) [state] {$v_1$};
                \node (q2) at (0,0) [state] {$v_0$};
                \path (q1) edge [loop right, right] node[right]{$b$}  (q1);
                \path (q1) edge [loop left, left] node[left]{$c$} (q1);
            \draw (q2) edge [ultra thick, below,postaction={decorate}] node[right]{$a$}  (q1)
      ;
       \end{tikzpicture}
                 \caption{}
        \label{fig:1core_inf_index1}
       \end{subfigure}
       \hfill 
       \begin{subfigure}[b]{.5\textwidth}
       \centering
       \includegraphics[width=\textwidth]{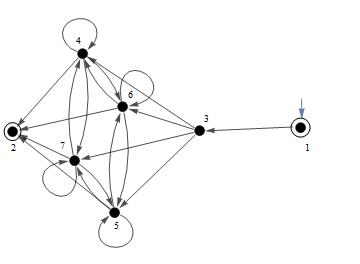}
       \caption{ }
        \label{fig:auto_inf_index1}
       \end{subfigure}
       \caption{(\ref{fig:1core_inf_index1}) is the core $\Delta_H$ of $H = \langle aba^{-1}, aca^{-1}\rangle$. (\ref{fig:auto_inf_index1}) is the Moore diagram of $\D_H$.}
        \end{figure}
    \item $H = \langle b^2, bab^{-1}a^{-1}, a^3\rangle$ is subgroup of $F_2$. 
    In this case, the core $\Delta_H$ is incomplete and it is a subgraph of $\Gamma$. So $H$ is an infinite index subgroup of $F_2$. See Figures (\ref{fig:core_inf_index2}) and (\ref{fig:auto_inf_index2}) for the core $\Delta_H$ and the diagram of the automaton $\D_H$, respectively. The $\deg(v_0) = 4.$ The adjacency matrix $M_{\D_H}$ of $G_{\A}$ is
              $$ M_{\D_H} = \left(
\begin{array}{ccccccccccccccc}
 0 & 0 & 0 & 0 & 0 & 0 & 1 & 1 & 0 & 0 & 0 & 0 & 0 & 1 & 1 \\
 0 & 0 & 0 & 0 & 0 & 0 & 0 & 1 & 0 & 0 & 0 & 0 & 0 & 1 & 1 \\
 0 & 0 & 0 & 0 & 0 & 0 & 1 & 0 & 0 & 0 & 0 & 0 & 0 & 1 & 1 \\
 0 & 0 & 0 & 0 & 0 & 0 & 1 & 1 & 0 & 0 & 0 & 0 & 0 & 1 & 0 \\
 0 & 0 & 0 & 0 & 0 & 0 & 1 & 1 & 0 & 0 & 0 & 0 & 0 & 0 & 1 \\
 0 & 1 & 0 & 0 & 0 & 0 & 0 & 0 & 0 & 0 & 0 & 0 & 0 & 0 & 0 \\
 0 & 0 & 0 & 0 & 0 & 0 & 0 & 0 & 1 & 0 & 0 & 0 & 0 & 0 & 0 \\
 0 & 0 & 0 & 0 & 0 & 1 & 0 & 0 & 0 & 0 & 0 & 1 & 0 & 0 & 0 \\
 0 & 0 & 1 & 0 & 0 & 0 & 0 & 0 & 0 & 0 & 0 & 1 & 0 & 0 & 0 \\
 0 & 0 & 1 & 0 & 0 & 1 & 0 & 0 & 0 & 0 & 0 & 0 & 0 & 0 & 0 \\
 0 & 0 & 0 & 0 & 0 & 0 & 0 & 0 & 0 & 1 & 0 & 0 & 0 & 0 & 0 \\
 0 & 0 & 0 & 0 & 0 & 0 & 0 & 0 & 0 & 0 & 0 & 0 & 1 & 0 & 0 \\
 0 & 0 & 0 & 1 & 1 & 0 & 0 & 0 & 0 & 0 & 0 & 0 & 0 & 0 & 0 \\
 0 & 0 & 0 & 1 & 0 & 0 & 0 & 0 & 0 & 0 & 1 & 0 & 0 & 0 & 0 \\
 0 & 0 & 0 & 0 & 1 & 0 & 0 & 0 & 0 & 0 & 1 & 0 & 0 & 0 & 0 \\
\end{array}
\right)$$
       Applying formula (\ref{transfer}) we get
       \begin{flalign}
        H(z) = \displaystyle\frac{1+z+3z^2+3z^3+5z^4+5z^5+6z^6+6z^7+4z^8+4z^9}{1+z+z^2-z^3-5z^4-13z^5-16z^6-20 z^7-12z^8-12z^9} && \nonumber
        \end{flalign}
        and $\alpha_H = 1.88233.$
        \begin{figure}[!htb]
    \begin{subfigure}[b]{0.5\textwidth}
        \centering
        \begin{tikzpicture}[scale=0.45,decoration={markings, mark= at position 0.5 with {\arrow{stealth}}}]
        \node[state] (q0) {$v_0$};
        \node[state, above right of=q0] (q1) {$v_1$};
        \node[state, above left of=q1] (q2) {$v_2$};
        \node[state, left of=q2] (q3) {$v_3$};
        \node[state, below of=q3] (q4) {$v_4$};
        \draw
               (q0) edge[ultra thick, above, right,postaction={decorate}] node{$a$} (q2)
               (q0) edge[ultra thick, left, above,postaction={decorate}] node{$b$} (q4)
               (q1) edge[ultra thick, below, right,postaction={decorate}] node{$a$}(q0)
               (q2) edge[below right, right,postaction={decorate}] node{$a$} (q1)
               (q2) edge[left, above,postaction={decorate}] node{$b$} (q3)
               (q4) edge[ultra thick, above, left,postaction={decorate}] node{$a$}(q3)
               (q4) edge[bend right, below,postaction={decorate}] node{$b$}(q0)
               ;
       \end{tikzpicture}
                 \caption{}
        \label{fig:core_inf_index2}
       \end{subfigure}
       \newline
       \begin{subfigure}[b]{.8\textwidth}
       \centering
       \includegraphics[width=\textwidth]{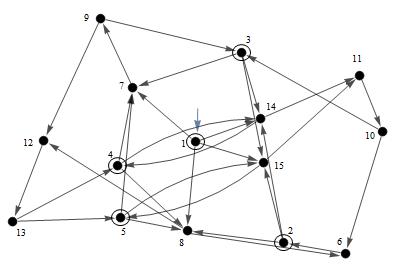}
       \caption{}
        \label{fig:auto_inf_index2}
       \end{subfigure}
       \caption{(\ref{fig:core_inf_index2}) The core graph $\Delta_H$ of $H = \langle b^2, bab^{-1}a^{-1},a^3 \rangle$. (\ref{fig:auto_inf_index2}) is the Moore diagram of $\D_H$.}
        \end{figure}
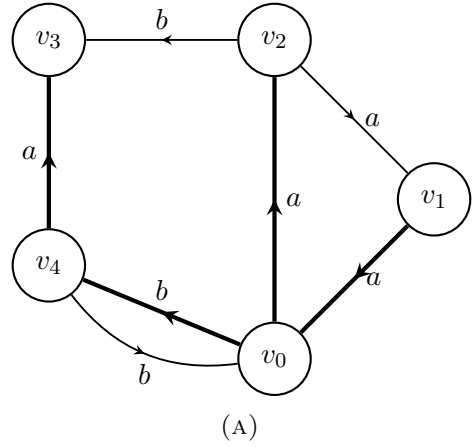
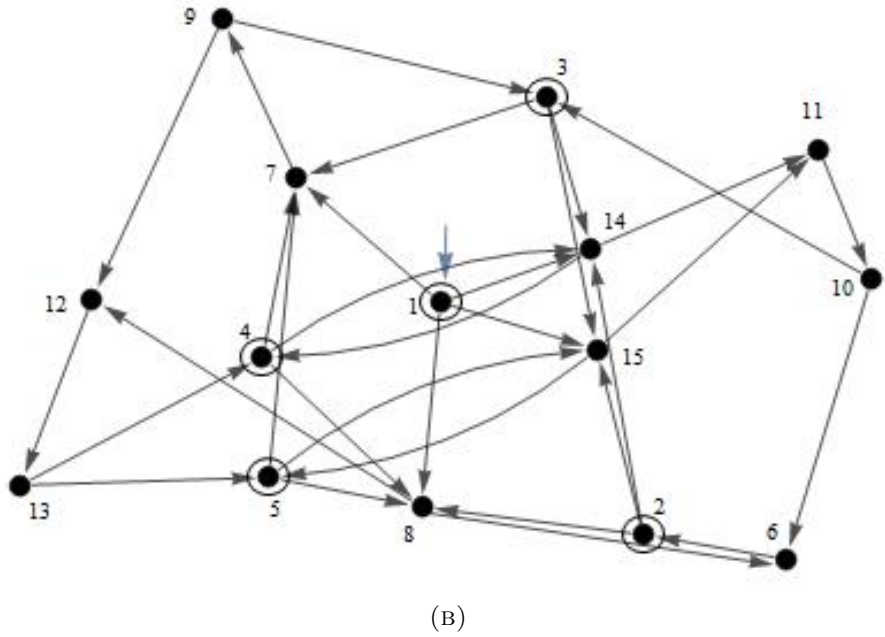
        \end{enumerate}
\end{enumerate}
\subsection{Computations of cogrowth using a Nielsen basis of $H$}\label{S:7.2}
In each of the examples below, we compute the growth series by solving the system (\ref{gri_system}) explained in the Section \ref{sec:computing_growth}. We shall consider the same set of examples that we have discussed in the previous section (\ref{section7.1}).
\begin{enumerate}
    \item Let $H$ be a finite index subgroup of $F_m$.
        \begin{enumerate}
            \item $H = \langle a^2, b, c, aba^{-1}, aca^{-1}\rangle$ is a subgroup of $F_3$. Solving the system 
    $$\left(\begin{array}{cc}
   B_1 & B_2 \\
   B_3 & B_4
\end{array}\right)\left(\begin{array}{c}
     Y_1 \\
     Y_2
\end{array}\right) = \left(\begin{array}{c}
     Z_1 \\ Z_2 
\end{array}\right),$$
 where  $$ B_1 = \left(
\begin{array}{ccccc}
 1-z^2 & 0 & -z^2 & -z^2 & -z^2 \\
 0 & 1-z^2 & -z^2 & -z^2 & -z^2 \\
 -z & -z & 1-z & 0 & -z \\
 -z & -z & 0 & 1-z & -z \\
 -z & -z & -z & -z & 1-z \\
\end{array}
\right), B_2 = \left(
\begin{array}{ccccc}
 -z^2 & -1 & -1 & -1 & -1 \\
 -z^2 & -z^2 & -z^2 & -z^2 & -z^2 \\
 -z & -z & -z & -z & -z \\
 -z & -z & -z & -z & -z \\
 0 & -z & -z & -z & -z \\
\end{array}
\right),$$
 $$ B_3 = \left(
\begin{array}{ccccc}
 -z & -z & -z & -z & 0 \\
 -z^3 & -z & -z^3 & -z^3 & -z^3 \\
 -z^3 & -z & -z^3 & -z^3 & -z^3 \\
 -z^3 & -z & -z^3 & -z^3 & -z^3 \\
 -z^3 & -z & -z^3 & -z^3 & -z^3 \\
\end{array}
\right), B_4 = \left(
\begin{array}{ccccc}
 1-z & -z & -z & -z & -z \\
 -z^3 & 1-z & 0 & -z & -z \\
 -z^3 & 0 & 1-z & -z & -z \\
 -z^3 & -z & -z & 1-z & 0 \\
 -z^3 & -z & -z & 0 & 1-z \\
\end{array}
\right),$$
$$Y_1 = \left(
\begin{array}{c}
 H_1^1 (z) \\
 H_1^{-1} (z) \\
 H_2^1 (z) \\
 H_2^{-1} (z) \\
 H_3^1 (z)
\end{array}
\right), Y_2 = \left(
\begin{array}{c}
 H_3^{-1} (z) \\
 H_4^1 (z) \\
 H_4^{-1} (z) \\
 H_5^1 (z) \\
 H_5^{-1} (z)
\end{array}
\right), Z_1 = \left(
\begin{array}{c}
 z^2 \\
 z^2 \\
 z \\
 z \\
 z \\
\end{array}
\right) \textnormal{ and } Z_2 = \left(
\begin{array}{c}
 z \\
 z^3 \\
 z^3 \\
 z^3 \\
 z^3 \\
\end{array}
\right)$$
we get $$ H(z) = \displaystyle \frac{(1+z)\left(1-4z+5 z^2\right)}{(1-5z)\left(1-2z+5z^2\right)}\textnormal{ and } \alpha_H =  5.$$
\item $H = \langle a^3,ab,ab^{-1},a^{-1}ba\rangle$ is a subgroup of $F_2$. Solving the system 
    $$BY = Z,$$
 where
 $$B = \left(
\begin{array}{cccccccc}
 1-z^3 & 0 & -z^3 & -z & -z^3 & -z & -z^3 & -z^3 \\
 0 & 1-z^3 & -z^3 & -z^3 & -z^3 & -z^3 & -z & -z \\
 -z^2 & -1 & 1-z^2 & 0 & -z^2 & -1 & -z^2 & -z^2 \\
 -z^2 & -z^2 & 0 & 1-z^2 & -z^2 & -z^2 & -z^2 & -z^2 \\
 -z^2 & -1 & -z^2 & -1 & 1-z^2 & 0 & -z^2 & -z^2 \\
 -z^2 & -z^2 & -z^2 & -z^2 & 0 & 1-z^2 & -z^2 & -z^2 \\
 -z & -z^3 & -z^3 & -z^3 & -z^3 & -z^3 & 1-z & 0 \\
 -z & -z^3 & -z^3 & -z^3 & -z^3 & -z^3 & 0 & 1-z \\
\end{array}
\right),$$
 $$Y =\left(
\begin{array}{c}
 H_1^1(z) \\
 H_1^{-1}(z) \\
 H_2^1(z) \\
 H_2^{-1}(z) \\
 H_3^1(z) \\
 H_3^{-1}(z) \\
 H_4^1(z) \\
 H_4^{-1}(z) \\
\end{array}
\right) \textnormal{ and } Z =\left(
\begin{array}{c}
 z^3 \\
 z^3 \\
 z^2 \\
 z^2 \\
 z^2 \\
 z^2 \\
 z^3 \\
 z^3 \\
\end{array}
\right)$$
we get
$$ H(z) = \displaystyle \frac{(1+z)(1-2z+5z^2-6z^3+9z^4)}{(1-3z) \left(1-z+3z^2\right) \left(1+3z+3z^2\right)} \textnormal{ and } \alpha_H = 3.$$
\end{enumerate}
\item Let $H$ be an infinite index subgroup of $F_m$.
\begin{enumerate}
\item $H = \langle aba^{-1}, aca^{-1}\rangle$ is a subgroup of $F_3$. Solving the system
    $$BY = Z,$$
 where  $$ B = \left(
\begin{array}{cccc}
 1-z & 0 & -z & -z \\
 0 & 1-z & -z & -z \\
 -z & -z & 1-z & 0 \\
 -z & -z & 0 & 1-z \\
\end{array}
\right),$$
$$Y =\left(
\begin{array}{c}
 H_1^1(z) \\
 H_1^{-1}(z) \\
 H_2^1(z) \\
 H_2^{-1}(z) \\
\end{array}
\right) \textnormal{ and } Z =\left(
\begin{array}{c}
 z^3 \\
 z^3 \\
 z^3 \\
 z^3 \\
\end{array}
\right)$$
we get
$$ H(z) = \displaystyle 1+ \frac{4z^3}{1-3z}\textnormal{ and } \alpha_H = 3. $$
\item $H = \langle b^2, bab^{-1}a^{-1}, a^3\rangle$ is a subgroup of $F_2$. Solving the system
    $$BY = Z,$$
 where  $$ B = \left(
\begin{array}{cccccc}
 1-z^2 & 0 & -z^2 & -1 & -z^2 & -z^2 \\
 0 & 1-z^2 & -z^2 & -z^2 & -z^2 & -z^2 \\
 -z^4 & -z^2 & 1-z^4 & 0 & -z^4 & -z^4 \\
 -z^4 & -z^4 & 0 & 1-z^4 & -z^4 & -z^2 \\
 -z^3 & -z^3 & -z & -z^3 & 1-z^3 & 0 \\
 -z^3 & -z^3 & -z^3 & -z^3 & 0 & 1-z^3 \\
\end{array}\right),$$
$$Y =\left(
\begin{array}{c}
 H_1^1(z) \\
 H_1^{-1}(z) \\
 H_2^1(z) \\
 H_2^{-1}(z) \\
 H_3^1(z) \\
 H_3^{-1}(z) \\
\end{array}
\right) \textnormal{ and } Z =\left(
\begin{array}{c}
 z^2 \\
 z^2 \\
 z^4 \\
 z^4 \\
 z^3 \\
 z^3 \\
\end{array}
\right)$$
we get
\begin{flalign}
        H(z) = \displaystyle\frac{1+z+3z^2+3z^3+5z^4+5z^5+6z^6+6z^7+4z^8+4z^9}{1+z+z^2-z^3-5z^4-13z^5-16z^6-20 z^7-12z^8-12z^9} && \nonumber
        \end{flalign}
        and $\alpha_H = 1.88233.$
\end{enumerate}
\end{enumerate}
\section{Acknowledgments}\label{S:Ackn}
RG was partially supported by Simons Foundation Collaboration Grant for Mathematicians, Award Number 527814. Also, RG  acknowledges the support of the Max Planck Institute for Mathematics in Bonn and Humboldt Foundation.
AS acknowledges United States-India Educational Foundation(USIEF), New Delhi and U.S. Department of State for the Fulbright Nehru Postdoctoral Research Fellowship, Award No.2479/FNPDR/2019. AS is grateful to S. P. Mandali, Pune and Government of Maharashtra for the sanction of study leave to undertake the fellowship. Also, AS acknowledges the generous hospitality of ICTS-TIFR, Bengaluru during the GGD 2017 program, where the initial idea of the project was discussed among the authors. The authors also would like to thank the anonymous referee for many suggestions that improved the exposition of the paper.
\bibliographystyle{plain}
\bibliography{Darbinyan_Grigorchuk_Shaikh_JGCC}
\end{document}